\documentclass{amsart}
\usepackage{amsrefs}

\DeclareMathOperator{\can}{can}
\DeclareMathOperator{\Scal}{Scal}
\DeclareMathOperator{\Weyl}{Weyl}

\DeclareMathOperator{\Eucl}{Eucl}
\renewcommand{\O}{\operatorname{O}}
\renewcommand{\o}{\operatorname{o}}
\newcommand{\<}{\left<}
\renewcommand{\>}{\right>}
\renewcommand{\(}{\left(}
\renewcommand{\)}{\right)}
\renewcommand{\[}{\left[}
\renewcommand{\]}{\right]}

\def \rr {\mathbb{R}}
\def \rn {\mathbb{R}^n}
\def \crit {2^\star}
\def \ue {u_\eps}

\def \xe {x_\eps}
\def \eps {\varepsilon}

\newtheorem{theorem}{Theorem}[section]

\newtheorem{proposition}[theorem]{Proposition}
\newtheorem{lemma}[theorem]{Lemma}

\begin{document}

\date{October 4th, 2012.}

\title{Sign-changing blow-up for scalar curvature type equations}

\author{Fr\'ed\'eric Robert}
\address{Fr\'ed\'eric Robert, Institut \'Elie Cartan, Universit\'e de Lorraine, BP 239, F-54506 Vand{\oe}uvre-l\`es-Nancy, France}
\email{frederic.robert@univ-lorraine.fr}

\author{J\'er\^ome V\'etois}
\address{J\'er\^ome V\'etois, Universit\'e de Nice Sophia Antipolis, Laboratoire J.-A. Dieudonn\'e, CNRS UMR 6621, Parc Valrose, F-06108 Nice Cedex 2, France}
\email{vetois@unice.fr}

\thanks{The authors are partially supported by the ANR grant ANR-08-BLAN-0335-01.}

\begin{abstract}
Given $(M,g)$ a compact Riemannian manifold of dimension $n\geq 3$, we are interested in the existence of blowing-up sign-changing families $(\ue)_{\eps>0}\in C^{2,\theta}(M)$, $\theta\in (0,1)$, of solutions to
$$\Delta_g \ue+h\ue=|\ue|^{\frac{4}{n-2}-\eps}\ue\hbox{ in }M\,,$$
where $\Delta_g:=-\hbox{div}_g(\nabla)$ and $h\in C^{0,\theta}(M)$ is a potential. We prove that such families exist in two main cases: in small dimension $n\in \{3,4,5,6\}$ for any potential $h$ or in dimension $3\leq n\leq 9$ when $h\equiv\frac{n-2}{4(n-1)}\Scal_g$. These examples yield a complete panorama of the compactness/noncompactness of critical elliptic equations of scalar curvature type on compact manifolds. The changing of the sign is necessary due to the compactness results of Druet~\cite{Dru1} and Khuri--Marques--Schoen~\cite{KhuMarSch}. 
\end{abstract}

\maketitle

\section{Introduction}
Let $(M,g)$ be a smooth compact Riemannian manifold of dimension $n\geq 3$. Given $\theta\in (0,1)$, we consider solutions $u\in C^{2,\theta}(M)$ to the equation
\begin{equation}\label{eq:stab}
\Delta_gu+h u=\left|u\right|^{\crit-2}u\quad\text{in }M\,,
\end{equation}
where $h\in C^{0,\theta}(M)$, $\Delta_g:=-\hbox{div}_g(\nabla)$ is the Laplace-Beltrami operator, and $\crit:=\frac{2n}{n-2}$. When $h\equiv \frac{n-2}{4(n-1)}\Scal_g$ ($\Scal_g$ being the scalar curvature of $(M,g)$), \eqref{eq:stab} is the Yamabe equation and rewrites
\begin{equation}\label{Eq1}
\Delta_gu+c_n\Scal_gu=\left|u\right|^{2^*-2}u\quad\text{in }M\,,
\end{equation}
where $c_n:= \frac{n-2}{4(n-1)}$. The conformal invariance of the Yamabe equation induces a dynamic that makes equations \eqref{eq:stab} and \eqref{Eq1} unstable. Taking inspiration from the terminology introduced by R. Schoen~\cite{Sch2}, we say that equation \eqref{eq:stab} is compact (resp. positively compact) if for any family $(q_\eps)_\eps\in (2,\crit]$ such that $q_\eps\to \crit$ when $\eps\to 0$ and for any family of functions (resp. positive functions) $(\ue)_\eps\in C^{2,\theta}(M)$ of solutions to
\begin{equation}
\Delta_g\ue+h\ue=\left|\ue\right|^{q_\eps-2}\ue\quad\text{in }M\label{eq:bis}
\end{equation}
for $\eps>0$ small, then a uniform bound on the Dirichlet energy $(\Vert\nabla\ue\Vert_2)_\eps$ implies the relative compactness of $(\ue)_\eps$ in $C^{2}(M)$, and therefore the convergence of a subfamily of $(\ue)_\eps$ in $C^{2}(M)$. Otherwise, we say that equation \eqref{eq:stab} is noncompact (resp. non positively-compact). A basic example of non compact equation is \eqref{Eq1} on the canonical sphere $(\mathbb{S}^n,\can)$: we refer to the second part of this section for (positive) compactness results for equations like \eqref{eq:stab}.

\medskip\noindent We say that a family $(\ue)_{\eps>0}\in C^{2,\theta}(M)$ blows-up when $\eps\to 0$ if $\lim_{\eps\to 0}\Vert\ue\Vert_\infty=+\infty$.  It is now well-known (see Struwe~\cite{struwe84} for a description in Sobolev spaces and Druet--Hebey--Robert~\cite{dhr} for a description in $C^0$) that noncompactness is described by bubbles. In the present paper, we investigate the existence of families $(\ue)_\eps\in C^{2,\theta}(M)$ of sign-changing blowing-up solutions to the equation
\begin{equation}\label{EqEPS}
\Delta_g\ue+h \ue=\left|\ue\right|^{\crit-2-\eps}\ue\quad\text{in }M\,,\, \eps>0.
\end{equation}
In the sequel, we say that a blowing-up family $(\ue)_\eps\in C^{2,\theta}(M)$ is of type $(u_0-B)$ if there exists $u_0\in C^{2,\theta}(M)$ and a bubble $(B_\eps)_{\eps}$ (see definition \eqref{def:bubble} below) such that 
\begin{equation}\label{model:sol:bubble}
\ue=u_0-B_\eps+\o(1),
\end{equation}
where $\lim_{\eps\to 0}\o(1)=0$ in $H_1^2(M)$, the completion of $C^\infty(M)$ for the norm $u\mapsto \Vert u\Vert_2+\Vert\nabla u\Vert_2$. Our first result is the following:
\begin{theorem}[\bf dimensions $3\leq n\leq 6$ and arbitrary potential]\label{Th2}\samepage
Let $\(M,g\)$ be a smooth compact Riemannian manifold of dimension $3\le n\le6$ and let $h\in C^{0,\theta}(M)$ ($\theta\in (0,1)$) be such that $\Delta_g+h$ is coercive. Assume that there exists a nondegenerate solution $u_0\in C^{2,\theta}(M)$ to equation \eqref{eq:stab}. In case $n=6$, assume in addition that $c_n\Scal_g-h<2u_0$ in $M$. Then for $\varepsilon>0$ small, equation \eqref{EqEPS} admits a sign-changing solution $u_\varepsilon$ of type $(u_0-B)$. In particular, the family $(\ue)_{\eps>0}$ blows up as $\varepsilon\to0$ and \eqref{eq:stab} is noncompact.
\end{theorem}
In full generality, it is not possible to construct positive blowing-up solutions to equation \eqref{EqEPS}. Indeed, in addition to the assumptions of Theorem \ref{Th2}, if we assume that $h<c_n\Scal_g$, then \eqref{eq:stab} is positively compact (Druet~\cite{Dru1} and the discussion below), and therefore any family of blowing-up solutions to \eqref{EqEPS} must be sign-changing. In the early reference~\cite{Ding}, Ding proved the existence of infinitely many nonequivalent solutions to \eqref{Eq1} on the canonical sphere, highlighting the diversity of the behavior of solutions to \eqref{eq:stab} depending on whether they are positive or negative.



\medskip\noindent The nondegeneracy assumption in Theorem \ref{Th2} is necessary. We refer to Proposition \ref{prop:cns}  in Section \ref{sec:hyp:nondege} for the proof of necessity. However, the nondegeneracy assumption of Theorem \ref{Th2} is generic in the sense that any degenerate solution to \eqref{eq:stab} can be approached by a solution of a slight perturbation of \eqref{eq:stab}. We refer to Proposition \ref{prop:gene:h} of Section \ref{sec:hyp:nondege} for the precise genericity statement.

\medskip\noindent The above theorem outlines a role of the geometry in dimension $n=6$. In higher dimension $n\geq 7$, the geometry of $(M,g)$ is more present. When the potential is strictly below the scalar curvature (that is $h<c_n\Scal_g$), equation \eqref{eq:stab} is compact for $n\geq 7$, at least in the locally conformally flat case (V\'etois~\cite{Vet}). Conversely, if $h(x_0)>c_n\Scal_g(x_0)$ for some $x_0\in M$, then under some additional nondegeneracy assumption, equation \eqref{eq:stab} is non-compact when $n\geq 7$ (see Pistoia--V\'etois~\cite{PisVet} for general results). Our second result is in the case $h\equiv c_n\Scal_g$:

\begin{theorem}[\bf dimensions $3\leq n\leq 9$ and $h\equiv c_n \Scal_g$]\label{Th1}
Let $\(M,g\)$ be a smooth compact Riemannian manifold of dimension $3\leq n\leq 9$ with positive Yamabe invariant. Assume that there exists a nondegenerate positive solution $u_0>0$ to the Yamabe equation \eqref{Eq1}. Assume that $h\equiv c_n \Scal_g$. Then for $\varepsilon>0$ small, equation \eqref{EqEPS} admits a sign-changing solution $u_\varepsilon$ of type $(u_0-B)$. In particular, the family $(\ue)_{\eps>0}$ blows up as $\varepsilon\to0$ and \eqref{Eq1} is noncompact.
\end{theorem}
Constructing positive blowing-up solutions is not possible in this context. Indeed, for $3\leq n\leq 9$ and except for the canonical sphere, the scalar curvature equation \eqref{Eq1} is positively compact (see Li--Zhu~\cite{LiZhu}, Druet~\cite{Dru2}, Marques~\cite{Mar}, Li--Zhang~\cites{LiZha1,LiZha3}, Khuri--Marques--Schoen~\cite{KhuMarSch}, and the discussion below). We refer also to Druet--Hebey~\cite{dh:analysis:pde}  and Druet--Hebey--V\'etois~\cite{dhv} for the extension of compactness issues to stability issues.


\medskip\noindent The restriction of the dimensions in Theorem \ref{Th1} is due to the geometry of the manifold. We refer to Subsection \ref{subsec:10} in Section \ref{sec:misc} for the extension of Theorem \ref{Th1} to dimension $n=10$ in general and in any dimension in the locally conformally flat case. 

\medskip\noindent Here again, it is natural to ask about the nondegeneracy assumption of a solution to the limit equation in Theorem \ref{Th1}: actually, it is both a necessary and a generic assumption. Concerning necessity, on the standard sphere (where all positive solutions to \eqref{Eq1} are degenerate), it is not possible to construct blowing-up solutions of type $(u_0-B)$, see Proposition \ref{prop:cns} in Section \ref{sec:hyp:nondege}. However, it is proved in Khuri--Marques--Schoen~\cite{KhuMarSch} that the nondegeneracy assumption is generic for the Yamabe equation \eqref{Eq1}, at least in dimensions $n\leq 24$, see Proposition \ref{prop:gene:geom} in Section \ref{sec:hyp:nondege}.


\medskip\noindent Here is a brief overview of the positive compactness results known so far for equations like \eqref{eq:stab}.

\medskip\noindent In 1987, Schoen~\cite{Sch2} adressed the question of positive compactness of equation \eqref{Eq1} for manifolds non conformally diffeomorphic to the canonical sphere $(\mathbb{S}^n,\can)$ (say aspherical manifolds). The known results are the following: positive compactness holds for aspherical locally conformally flat manifolds (Schoen~\cites{Sch2,Sch3}) and for arbitrary aspherical manifolds of dimension $3\leq n\leq 24$ (Li--Zhu~\cite{LiZhu} ($n=3$), Druet~\cite{Dru2} ($n\leq 5$), Marques~\cite{Mar} ($n\leq 7$), Li--Zhang~\cites{LiZha1,LiZha3,LiZha4} ($n\leq 11$), Kuhri--Marques--Schoen~\cite{KhuMarSch} ($n\leq 24$)). But positive compactness does not hold in general in dimension $n\geq 25$ (There are blowing-up examples by Brendle~\cite{Bre} and Brendle--Marques~\cite{BreMar}). Combining these results with Theorem \ref{Th2}, we get that equation \eqref{Eq1} is positively compact, but not compact, at least when $3\leq n\leq 9$.

\medskip\noindent When $h\not\equiv c_n\Scal_g$, the situation is different. When $h<c_n\Scal_g$, Druet~\cite{Dru1} proved that  \eqref{eq:bis} is positively compact in dimension $n\geq 3$ (see also Li--Zhu~\cite{LiZhu} and Druet--Hebey--V\'etois~\cite{dhv} for $n=3$). Conversely, in dimension $n\geq 4$, Micheletti-Pistoia--V\'etois~\cite{MicPisVet} and Pistoia--V\'etois~\cite{PisVet} proved that if $h$ is above $c_n\Scal_g$ somewhere, then, under some some nondegeneracy assumption, equation \eqref{eq:stab} is not positively compact. On the canonical sphere, there are blowing-up positive solutions with arbitrarily high energy when $h\equiv Cte>c_n\Scal_{\can}$ (Chen--Wei--Yan~\cite{chenweiyan} for $n\geq 5$). We refer to Esposito--Pistoia--V\'etois~\cite{EspPisVet} for blowing-up positive solutions in case of a potential $h$ depending on $\varepsilon$ and approaching $c_n\Scal_g$, and to Hebey--Wei~\cite{hw} for the construction of multi-peak solutions on the three-sphere with a potential approaching constants arbitrarily larger than the scalar curvature. Here again, combining Druet~\cite{Dru1} and Theorem \ref{Th1} yields the following: when $h<c_n\Scal_g$ and $3\leq n \leq 5$, equation \eqref{eq:stab} is positively compact, but not compact.


\medskip\noindent The proofs of Theorems~\ref{Th2} and~\ref{Th1} rely on a Lyapunov--Schmidt reduction. Over the past two decades, there has been intensive developments in Lyapunov--Schmidt reductions applied to critical elliptic equations. In addition to the references in the geometric context of a Riemannian manifold cited above, an early reference for single-bubble solutions is Rey~\cite{Rey}. Possible references on the construction of blowing-up solutions to equations like \eqref{EqEPS} by glueing a fixed function to bubbles are del Pino--Musso--Pacard--Pistoia~\cites{delPMusPacPis1,delPMusPacPis2} and Guo--Li--Wei~\cite{GuoLiWei} (for the Yamabe equation on the canonical sphere) and Wei--Yan~\cite{WeiYan} (for a Lazer-McKenna type problem). The list of constributions above does not pretend to exhaustivity: we refer to the references of the above papers and also to the monograph~\cite{AmbMal} by Ambrosetti--Malchiodi for further bibliographic complements. Our paper is organized as follows. In Section \ref{sec:misc}, we discuss extensions and generalizations of the above theorems. In Section \ref{sec:hyp:nondege}, we discuss the nondegeneracy assumption. The finite dimensional reduction is performed in Section \ref{sec:finite}. The reduced problem is studied in Section \ref{sec:reduced}. Theorems \ref{Th2} and \ref{Th1} are proved in Section \ref{sec:pfs}. The proof of the error estimate is postponed to Section \ref{sec:error}.

\medskip\noindent{\bf Acknowledgements:} The authors express their gratitude to Emmanuel Hebey and Lionel B\'erard-Bergery for fruitful comments on this work.

\section{Miscellaneous extensions}\label{sec:misc}

\subsection{About the critical dimension $n=10$ in Theorem \ref{Th1}}\label{subsec:10}
As mentioned in the introduction, the method developed here fails to produce blowing-up solutions to \eqref{EqEPS} in higher dimension. Indeed, in dimensions $n\geq 7$, a term involving the Weyl tensor appear in the Taylor expansion \eqref{eq:exp:10} of the Lyapunov--Schmidt functional. In dimension $n<10$, this term is dominated by the contribution of $u_0$. In dimension $n=10$, there is a competition between the Weyl tensor and $u_0$, and one gets the following result:

\begin{theorem}[\bf dimension $n=10$ and $h\equiv c_n \Scal_g$]\label{Th:dim10}
Let $\(M,g\)$ be a smooth compact Riemannian manifold of dimension $n=10$ with positive Yamabe invariant. Assume that there exists a nondegenerate positive solution to the Yamabe equation \eqref{Eq1}. Assume that $h\equiv c_n \Scal_g$ and that $u_0>\frac{5}{567}|\Weyl_g|_g^2$. Then for $\varepsilon>0$ small, equation \eqref{EqEPS} admits a sign-changing solution $u_\varepsilon$ of type $(u_0-B)$. In particular, the family $(\ue)_{\eps>0}$ blows up as $\varepsilon\to0$.
\end{theorem}
In dimension $n>10$, the Weyl tensor dominates the contribution of $u_0$, and it is not possible to produce blowing-up solutions in general (see the explicit Taylor expansion \eqref{eq:exp:10} in Section \ref{sec:pfs}). However, in the locally conformally flat case, that is when the Weyl tensor vanishes (at least in dimension larger than four), one gets the following result:
\begin{theorem}[\bf the locally conformally flat case in any dimension]\label{Th:lcf}
Let $\(M,g\)$ be a smooth compact Riemannian manifold of dimension $n\ge3$ with positive Yamabe invariant. Assume that there exists a nondegenerate positive solution to the Yamabe equation \eqref{Eq1}. Assume that $(M,g)$ is locally conformally flat and that $h\equiv c_n \Scal_g$. Then for $\varepsilon>0$ small, equation \eqref{EqEPS} admits a sign-changing solution $u_\varepsilon$ of type $(u_0-B)$. In particular, the family $(\ue)_{\eps>0}$ blows up as $\varepsilon\to0$.
\end{theorem}
Examples of manifolds and metrics satisfying the hypothesis of Theorem \ref{Th:lcf} are in Proposition \ref{prop:ex:lcf}. As stated in Theorems \ref{Th:dim10} and \ref{Th:lcf}, the solutions we construct change sign. Here again, it is natural to ask if there exist positive blowing-up solutions to \eqref{EqEPS} under the assumptions of Theorems \ref{Th:dim10} and \ref{Th:lcf}. The answer is negative. Indeed, it follows from positive compactness theorems (Khuri--Marques--Schoen~\cite{KhuMarSch} and Schoen~\cites{Sch2,Sch3}) that positive blowing-up solutions to equation \eqref{EqEPS} do not exist in the locally conformally flat and aspherical case. A consequence of the above results is that the Yamabe equation \eqref{Eq1} is positively compact, but not compact in the context of Theorems \ref{Th:dim10} and \ref{Th:lcf}.

\subsection{Positive blowing-up solutions in dimension $n=6$}\label{subsec:6}
In this subsection, we focus on positive solutions to \eqref{EqEPS}. A direct offshot of the techniques developed for the proof of Theorem \ref{Th2} yields the following result:

\begin{theorem}[\bf positive solutions in dimension $n=6$]\label{Th:dim6}
Let $\(M,g\)$ be a smooth compact Riemannian manifold of dimension $n=6$ and let $h\in C^{0,\theta}(M)$ ($\theta\in (0,1)$) be such that $\Delta_g+h$ is coercive. Assume that there exists a nondegenerate solution $u_0\in C^{2,\theta}(M)$ to equation \eqref{eq:stab} and that
$$h-c_6\Scal_g>2u_0>0\quad\text{in }M\,.$$
Then for $\varepsilon>0$ small, equation \eqref{EqEPS} admits a positive solution $u_\varepsilon>0$ such that $\ue=u_0+B_{\eps}+o(1)$, where $(B_{\eps})_\eps$ is a bubble and $\lim_{\eps\to 0}o(1)=0$ in $H_1^2(M)$.
\end{theorem}
This result is a complement to a specific 6-dimensional result: Druet (\cite{Dru1} and private communication) showed that blow-up for positive solutions to \eqref{EqEPS} with bounded energy necessarily occurs at points $x\in M$ such that $(h-c_6\Scal_g)(x)\geq 2u_0(x)$. Dimension six is critical when considering positive blowing-up solutions with nontrivial weak limit $u_0>0$. More precisely, the blow-up analysis shows that there is  balance between the contributions of $u_0$ and $h-c_n\Scal_g$: one of the terms dominates the other when $n\neq 6$, and they compete at the same growth when $n=6$. We refer to the Taylor expansion \eqref{Pr2Eq1} and to~\cite{Dru1} to outline this phenomenon. We refer to Druet--Hebey~\cite{dh:analysis:pde} for an extensive discussion on dimension six.

\subsection{Prescription of the blow-up point}\label{subsec:presc}
The above theorems show the existence of blowing-up families of solutions, but the blow-up point (that is the point where the bubble is centered) is not prescribed. The only information obtained from the construction is that blow-up occurs at a minimum point of $u_0$ (for Theorems \ref{Th2} and \ref{Th1} when $n\neq 6$) or of $u_0-(c_n\Scal_g-h)/2$ (for Theorem \ref{Th2} when $n=6$). Prescribing the location of the blow-up point of the bubbles requires additional informations. We define $\Phi:M\to\rr$ as follows:
$$\Phi:=\left\{\begin{array}{ll}
u_0+\frac{1}{2}(h-c_n\Scal_g){\bf 1}_{n=6} &\hbox{in Theorems \ref{Th2}, \ref{Th1} and \ref{Th:lcf}}\\
u_0-\frac{5}{567}|\Weyl_g|_g^2 &\hbox{in Theorem \ref{Th:dim10}}\\
\frac{1}{2}(h-c_6\Scal_g)-u_0&\hbox{in Theorem \ref{Th:dim6}}.
\end{array}\right.$$
Our prescription result is the following:
\begin{theorem}[\bf Prescription of the blow-up point]\label{Th:presc} In addition to the hypothesis in Theorems \ref{Th2}, \ref{Th1}, \ref{Th:dim10}, \ref{Th:lcf}, and \ref{Th:dim6}, assume that there exists $\xi_0\in M$ which is a strict local minimum point of $\Phi$ with $\Phi(\xi_0)>0$. Then the conclusions of the above theorems hold with the extra information that the bubbles are centered at a family $(\xi_\eps)_{\eps}\in M$ such that $\lim_{\eps\to 0}\xi_\eps=\xi_0$.
\end{theorem}
In case $h\in C^1(M)$ and there exists $\xi_0\in M$ which is a $C^1-$stable critical point of $\Phi$ with $\Phi(\xi_0)>0$, the same conclusion holds with the convergence \eqref{Pr2Eq1} holding in $C^1$.

\section{Discussion on the degenerate case}\label{sec:hyp:nondege}
In the sequel, we say that $(B_\eps)_\eps$ is a bubble if there exists a family $(\xe)_\eps\in M$ and a family $(\mu_\eps)_\eps\in \rr_{>0}$ such that $\lim_{\eps\to 0}\mu_\eps=0$ and
\begin{equation}\label{def:bubble}
B_\eps(x):=\left(\frac{\sqrt{n(n-2)}\mu_\eps}{\mu_\eps^2+d_g(x,x_\eps)^2}\right)^{\frac{n-2}{2}}\hbox{ for all }x\in M\,.
\end{equation}
In this situation, we say that the bubble is centered at $(x_\eps)_{\eps}$. We say that a solution $u_0\in C^{2,\theta}(M)$ to 
\begin{equation}\label{Th2Eq}
\Delta_gu_0+h u_0=u_0^{\crit-1}\quad\text{in }M
\end{equation}
is nondegenerate if the linearization of the equation has a trivial kernel, that is
\begin{equation}\label{def:K}
K_{h,u_0}:=\left\{\varphi\in C^{2,\theta}(M)/\; \Delta_g \varphi+h\varphi=(\crit-1)|u_0|^{\crit-2}\varphi\right\}=\{0\}.
\end{equation}
Theorems \ref{Th2} and \ref{Th1} require the assumption that $u_0$ is a nondegenerate solution to \eqref{Th2Eq}. In this section, we prove that this is a necessary assumption, and that it is generic.

\subsection{The conformal geometric equation and necessity of the nondegeneracy assumption}
Let $(M,g)$ be a compact Riemannian manifold of dimension $n\geq 3$ with positive Yamabe invariant. Up to a conformal change of metric, it follows from the resolution of the Yamabe problem that we can assume that the scalar curvature $\Scal_g$ is a positive constant, and we consider $u_{0,g}:=(c_n\Scal_g)^{1/(\crit-2)}$ the only positive constant solution to the Yamabe equation
\begin{equation}\label{eq:cst:yam}
\Delta_g u_{0,g}+c_n\Scal_g u_{0,g}=u_{0,g}^{\crit-1}\quad\hbox{in }M\,.
\end{equation}
As is easily checked, in this situation,
$$K_{c_n\Scal_g, u_{0,g}}=\left\{\varphi\in C^2(M)/\, \Delta_g \varphi=\frac{\Scal_g}{n-1}\varphi\right\},$$
where the kernel is defined in \eqref{def:K}. Therefore 
$$u_{0,g}\hbox{ is a nondegenerate solution to \eqref{eq:cst:yam}}\Leftrightarrow\; \frac{\Scal_g}{n-1}\not\in\,\hbox{Spec }(\Delta_g),$$
where Spec $(\Delta_g)$ is the nonnegative spectrum of $\Delta_g$. We define the Yamabe invariant by
\begin{equation}\label{def:yam:invar}
\mu_{[g]}(M):=\inf_{g'\in [g]}
\frac{\int_M \Scal_{g'}\, dv_{g'}}{\hbox{Vol}_{g'}(M)^{\frac{n-2}{n}}}\, ,
\end{equation}
where $[g]$ is the conformal class of $g$ and $dv_g$ is the Riemannian element of volume. The Yamabe invariant $\mu_{[g]}(M)$ is positive iff the operator $\Delta_{g'}+c_n\Scal_{g'}$ is coercive for all $g'\in [g]$. It is well known that if $g$ is a Yamabe metric (that is a minimizer of the Yamabe functional \eqref{def:yam:invar}), one has that $\frac{\Scal_g}{n-1}\leq \lambda_1(\Delta_g)$, the first nonzero eigenvalue of $\Delta_g$. Note that equality is achieved on the canonical sphere $(\mathbb{S}^n,\can)$. More generally, any positive solution to the Yamabe equation on the canonical sphere is a Yamabe metric and is degenerate.

\medskip\noindent The following result shows that the conclusion of Theorems \ref{Th2} and \ref{Th1} do not hold on the standard sphere (where positive solutions to the Yamabe equations are all degenerate):

\begin{proposition}\label{prop:cns} There does not exist any family of functions $(\ue)_\eps\in C^{2,\theta}(\mathbb{S}^n)$ of type $(u_0-B)$ to the equation
\begin{equation}\label{eq:ue}
\Delta_{\can}\ue +c_n \Scal_{\can}\ue=|\ue|^{\crit-2-\eps}\ue\quad\hbox{in }M
\end{equation}
for all $\eps\in (0,\crit)$.
\end{proposition}

\proof[Proof] We argue by contradiction and assume the existence of a family $(\ue)_\eps\in C^{2,\theta}(\mathbb{S}^n)$ of the form \eqref{model:sol:bubble} of solutions to equation \eqref{eq:ue}. Multiplying \eqref{eq:ue} by the bubble $B_\eps$ and integrating by parts yield $\mu_\eps^\eps\to 1$ when $\eps\to 0$. Fix $\phi\in \Lambda_1(\mathbb{S}^n,\can)$, the set of eigenfunctions of $\lambda_1(\Delta_{\can})=n$, the first nonzero eigenvalue of $\Delta_{\can}$ on $\mathbb{S}^n$: indeed, see for instance Berger--Gauduchon--Mazet~\cite{bgm}, we have that $\Lambda_1(\mathbb{S}^n,\can)=\{l_{|\mathbb{S}^n}/ \, l:\mathbb{R}^{n+1}\to\mathbb{R}\hbox{ linear}\}$. It follows from Kazdan--Warner~\cite{kazdanwarner} that
$$\int_{\mathbb{S}^n}\Delta_{\can}\ue \langle\nabla\phi,\nabla\ue\rangle_{\can}\, dv_{\can}=\frac{n-2}{2n}\int_{\mathbb{S}^n}\Delta_{\can}\phi |\nabla \ue|_{\can}^2\, dv_{\can}\,.$$
Using equation \eqref{eq:ue} and integrating by parts yields
$$\eps\int_{\mathbb{S}^n}\phi |\ue|^{\crit-\eps}\, dv_{\can}=0\,.$$
Letting $\eps\to 0$ and using \eqref{def:bubble}, \eqref{model:sol:bubble}, and \eqref{Eq10} yields, up to a subsequence, 
$$\int_{\mathbb{S}^n}\phi u_0^{\crit}\, dv_{\can}+\left(\int_{\rn}U^{\crit}\, dx\right) \phi(x_0)=0\,,$$
where $x_0\in \mathbb{S}^n$. Passing to the weak limit in \eqref{eq:ue} when $\eps\to 0$ implies that $u_0$ is a positive solutions to the Yamabe equation on $\mathbb{S}^n$. It then follows from Obata~\cite{obata} that $\int_{\mathbb{S}^n} u_0^{\crit}\, dv_{\can}=\int_{\rn}U^{\crit}\, dx$. Taking $\phi\in \Lambda_1(\mathbb{S}^n,\can)$ such that $\min\, \phi=\phi(-x_0)\neq 0$ in the above equation yields a contradiction since $\phi\not\equiv 0$. This ends the proof of Proposition \ref{prop:cns}.
\endproof

\subsection{Genericity of the nondegeneracy assumption} The following proposition shows that the nondegeneracy hypothesis of Theorem \ref{Th2} is generic:
\begin{proposition}\label{prop:gene:h} Let $h\in C^{0,\theta}(M)$ and let $u_0\in C^{2,\theta}(M)$ be a positive solution to $\Delta_g u_0+hu_0=u_0^{\crit-1}$. Fix $\nu>0$. Then there exist $\tilde{h}_\nu\in C^{0,\theta}(M)$ and $\tilde{u}_{0,\nu}\in C^{2,\theta}(M)$ such that $\Vert h-\tilde{h}_\nu\Vert_{C^{0,\theta}}+\Vert u_0-\tilde{u}_{0,\nu}\Vert_{C^{2,\theta}}<\nu$ and $\tilde{u}_{0,\nu}>0$ is a nondegenerate solution to $\Delta_g \tilde{u}_{0,\nu}+\tilde{h}_\nu\tilde{u}_{0,\nu}=\tilde{u}_{0,\nu}^{\crit-1}$ in $M$.
\end{proposition}

\begin{proof} We define
$$\mu_\eta:=\inf_{u\in H_1^2(M)\setminus\{0\}}\frac{\int_M(|\nabla u|_g^2+(h-(\crit-1)u_0^{\crit-2}-\eta)u^2)\, dv_g}{\left(\int_M |u|^{\crit}\, dv_g\right)^{\frac{2}{\crit}}}$$
for all $\eta\geq 0$. Testing the functional on $u_0$ yields $\mu_{\eta}<0$ for all $\eta\geq 0$. As is easily checked, $\lim_{\eta\to 0}\mu_\eta=\mu_0<0$. Standard variational arguments yield the existence of a positive minimizer $w_\eta\in C^{2,\theta}(M)$ for $\mu_\eta$ such that $\Delta_g w_\eta+(h-(\crit-1)u_0^{\crit-2}-\eta)w_\eta=-(\crit-2)w_\eta^{\crit-1}$ in $M$ for all $\eta>0$; moreover, the family $(w_\eta)_{\eta\geq 0}$ is relatively compact in $C^{2}(M)$. Since $u_0$ is the only positive solution to the equation $\Delta_g v+(h-(\crit-1)u_0^{\crit-2})v=-(\crit-2)v^{\crit-1}$ (let $w$ be the quotient of two positive solutions and estimate $\Delta_g w$ at extremal points of $w$), one gets that $\lim_{\eta\to 0}w_\eta=u_0$ in $C^2(M)$, and then $C^{2,\theta}(M)$ by elliptic regularity. One defines $h_\eta:=h-(\crit-1)(u_0^{\crit-2}-w_\eta^{\crit-2})-\eta$. Then $\Delta_g w_\eta+h_\eta w_\eta=w_\eta^{\crit-1}$ in $M$ and spectral theory yields the existence of $\eta_0>0$ such that $K_{h_\eta,w_\eta}=\{0\}$ for all $\eta\in (0,\eta_0)$. The conclusion of the proposition follows from taking $\tilde{u}_{0,\nu}:=w_\eta$ and $\tilde{h}_\nu:=h_\eta$ for $\eta>0$ small enough.
\end{proof}

\noindent We now focus on the geometric case, that is Theorem \ref{Th1}. We adopt the terminology of Khuri--Marques--Schoen~\cite{KhuMarSch}: given $M$ a compact manifold, $g_0$ a background Riemannian metric on $M$, and $\omega$ a volume form on $M$, to each class  $c\in {\mathcal C}:=\{\hbox{Conformal classes of Riemannian metrics on }M\}$, we associate the unique metric $g\in c$ for which the Riemannian $n-$volume form is $\omega$. The $C^{k,\theta}$--distance between two classes is defined as the $C^{k,\theta}$-distance between their representatives via this analogy. The following result holds:

\begin{proposition}[Khuri--Marques--Schoen~\cite{KhuMarSch}]\label{prop:gene:geom}
There exists ${\mathcal O}\subset {\mathcal C}$ an open dense set such that for all $c\in {\mathcal O}$, there exists exactly a finite nonzero number of metrics $g\in c$ (up to homothetic transformations) such that the constant positive function $u_{0,g}$ is nondegenerate.
\end{proposition} 
\noindent In other words, up to a perturbation in the conformal class, the hypothesis of Theorem \ref{Th1} holds.

\subsection{A family of nondegenerate geometric solutions in the locally conformally flat case}
We exhibit here a situation in which the nondegeneracy assumption is satisfied for the geometric equation in the locally conformally flat case (see Theorem \ref{Th:lcf}). For all $k\geq 1$ and $t>0$, $(\mathbb{S}^k(t),\can)$ is the canonical sphere ot radius $t$ in $\rr^{k+1}$:
\begin{proposition}\label{prop:ex:lcf}
Let $M_r:=\mathbb{S}^1(r)\times \mathbb{S}^{n-1}$ be endowed with its canonical product metric $g_r$, where $r>0$. Then $(M_r,g_r)$ is locally conformally flat with positive constant scalar curvature. Moreover, for any $r\not\in \{\frac{i}{\sqrt{n-2}}/\, i\in \mathbb{Z}_{>0}\}$, the positive constant solution to the Yamabe equation is nondegenerate.
\end{proposition}

\begin{proof} Recall that on the canonical sphere $\mathbb{S}^k(t)$, the spectrum of the Laplacian is $\big\{\frac{i(k+i-1)}{t^2}/\, i\in \mathbb{Z}_{\geq 0}\big\}$ (see~\cite{bgm}). Then the spectrum of $\Delta_{g_r}$ is $\{\frac{i^2}{r^2}+j(n-2+j)/\, i,j\in \mathbb{Z}_{\geq 0}\}$. Independently, the scalar curvature is $\Scal_{g_r}:=(n-1)(n-2)$. As a consequence,
$$\frac{\Scal_{g_r}}{n-1}\not\in \hbox{ Spec}(\Delta_{g_r})\;\Leftrightarrow\; r\not\in \left\{\frac{i}{\sqrt{n-2}}/\, i\in \mathbb{Z}_{>0}\right\}.$$
In addition, it is standard that the product of a one-dimensional circle with a space form is locally conformally flat. \end{proof}

\section{Finite dimensional reduction}\label{sec:finite}
Since the operator $\Delta_g+h$ is coercive, the Sobolev space $H^2_1\(M\)$ is endowed with the scalar product $\<\cdot,\cdot\>_h$ defined by 
\begin{equation}\label{Eq7}
\<u,v\>_h=\int_M\<\nabla u,\nabla v\>_gdv_g+\int_Mhuvdv_g
\end{equation}
for all $u,v\in H_1^2(M)$. We let $\|\cdot\|_h$ be the norm induced by $\<\cdot,\cdot\>_h$: this norm is equivalent to the standard norm on $H_1^2(M)$. We let $i^*:L^{\frac{2n}{n+2}}\(M\)\rightarrow H^2_1\(M\)$ be such that for any $w$ in $L^{\frac{2n}{n+2}}\(M\)$, the function $u=i^*\(w\)$ in $H^2_1\(M\)$ is the unique solution of the equation $\Delta_gu+hu=w$ in $M$. We then rewrite equation \eqref{EqEPS} as
\begin{equation}\label{Eq8}
u=i^*\(f_\varepsilon\(u\)\),\quad u\in H^2_1\(M\),
\end{equation}
where $f_\varepsilon\(u\):=\left|u\right|^{2^*-2-\varepsilon}u$.\par
\medskip\noindent  In case $(M,g)$ is locally conformally flat, it follows from the compactness of $M$ that there exists $r_0\in (0, i_g(M))$ (where $i_g(M)>0$ is the injectivity radius of $(M,g)$) such that for any point $\xi$ in $M$, there exists $\varLambda_\xi\in C^\infty(M)$ such that the conformal metric $g_\xi=\varLambda_\xi^{4/\(n-2\)}g$ is flat in $B_\xi\(r_0\)$ and $i_{g_\xi}(M)> r_0$. As is easily seen, the functions $\varLambda_\xi$ can be chosen smooth with respect to $\xi$ and such that $\varLambda_\xi\(\xi\)=1$. If the manifold is not locally conformally flat, then we let $\varLambda_\xi\(x\)=1$ for all points $x$ and $\xi$ in $M$, and we fix $r_0\in (0,i_g(M))$ arbitrarily. We let $\chi$ be a smooth cutoff function such that $0\le\chi\le1$ in $\mathbb{R}$, $\chi=1$ in $\[-r_0/2,r_0/2\]$, and $\chi=0$ in $\mathbb{R}\backslash\(-r_0,r_0\)$. For any positive real number $\delta$ and any point $\xi$ in $M$, we define the function $W_{\delta,\xi}$ on $M$ by
\begin{equation}\label{Eq9}
W_{\delta,\xi}\(x\):=\chi\(d_{g_\xi}\(x,\xi\)\)\varLambda_\xi\(x\)\delta^{\frac{2-n}{2}}U\(\delta^{-1}\exp_\xi^{-1}\(x\)\),
\end{equation}
where $d_{g_\xi}$ is the geodesic distance on $M$ associated with the metric $g_\xi$, the exponential map is taken with respect to the same metric $g_\xi$ and
\begin{equation}\label{Eq10}
U\(x\):=\(\frac{\sqrt{n\(n-2\)}}{1+\left|x\right|^2}\)^{\frac{n-2}{2}}
\end{equation}
for all $x\in\rn$. For any positive real number $\delta$, the function $U_\delta\(x\)=\delta^{\frac{2-n}{2}}U\(\delta^{-1}x\)$ satisfies the equation $\Delta_{\Eucl}U_\delta=U_\delta^{\crit-1}$ where $\Delta_{\Eucl}$ is the Laplace operator associated with the Euclidean metric. Moreover, by Bianchi--Egnell~\cite{BiaEgn}, any solution in $v\in D_1^2\(\mathbb{R}^n\)$ (the completion of $C^\infty_c(\rn)$ for the norm $\Vert u\Vert_{D_1^2}:=\Vert\nabla u\Vert_2$) of the linear equation $\Delta_{\Eucl}v=\(2^*-1\)U^{2^*-2}v$ is a linear combination of the functions 
\begin{equation}\label{Eq11}
V_0\(x\):=\frac{\left|x\right|^2-1}
{\(1+\left|x\right|^2\)^{\frac{n}{2}}}\quad\text{and}\quad V_i\(x\):=\frac{x_i}
{\(1+\left|x\right|^2\)^{\frac{n}{2}}}
\end{equation}
for all $i=1,\dotsc,n$ and $x\in\rn$. For any positive real number $\delta$ and any point $\(\xi,\omega\)$ in $TM$, we define the functions $Z_{\delta,\xi}$ and $Z_{\delta,\xi,\omega}$ in $M$ by
\begin{align}
&Z_{\delta,\xi}\(x\):=\chi\(d_{g_\xi}\(x,\xi\)\)\varLambda_\xi(x)\delta^{\frac{n-2}{2}}\frac{d_{g_\xi}\(x,\xi\)^2-\delta^2}
{\(\delta^2+d_{g_\xi}\(x,\xi\)^2\)^{\frac{n}{2}}}\,,\label{Eq12}\\
&Z_{\delta,\xi,\omega}\(x\):=\chi\(d_{g_\xi}\(x,\xi\)\)\varLambda_\xi(x)\delta^{\frac{n}{2}}\frac{\<\exp_\xi^{-1}x,\omega\>_{g_\xi}}
{\(\delta^2+d_{g_\xi}\(x,\xi\)^2\)^{\frac{n}{2}}}\label{Eq13}
\end{align}
for all $x\in M$. We then let $\varPi_{\delta,\xi}$ and $\varPi^\perp_{\delta,\xi}$ be the projections of the Sobolev space $H^2_1\(M\)$ onto the respective closed subspaces
\begin{equation*}
K_{\delta,\xi}:=\left\{\; \lambda Z_{\delta,\xi}+Z_{\delta,\xi,\omega}\,/\;\lambda\in\mathbb{R}\;\text{and}\;\omega\in T_\xi M\right\}
\end{equation*}
and
\begin{equation}\label{Eq15}
K^\perp_{\delta,\xi}:=\left\{\phi\in H^2_1\(M\)\; /\;\<\phi,Z_{\delta,\xi}\>_h=0\;\text{and}\;\<\phi,Z_{\delta,\xi,\omega}\>_h=0\quad\text{for all }\omega\in T_\xi M\right\},
\end{equation}
where the scalar product $\<\cdot,\cdot\>_h$ is as in \eqref{Eq7}. 

\medskip\noindent Recall that $u_0\in C^{2,\theta}(M)$ is a nondegenerate positive solution to equation \eqref{Th2Eq}. We construct solutions of type $(u_0-B)$ to equations \eqref{EqEPS}, or equivalently \eqref{Eq8}, like 
$$u_\varepsilon:=u_0-W_{\delta_\varepsilon\(t_\varepsilon\),\xi_\varepsilon}+\phi_{\delta_\varepsilon\(t_\varepsilon\),\xi_\varepsilon}\,,\quad\text{with }\delta_\varepsilon\(t_\varepsilon\):=t_\varepsilon\varepsilon^{\frac{2}{n-2}},$$
where $W_{\delta_\varepsilon\(t_\varepsilon\),\xi_\varepsilon}$ is as in \eqref{Eq9}, $\phi_{\delta_\varepsilon\(t_\varepsilon\),\xi_\varepsilon}$ is a function in $K^\perp_{\delta_\varepsilon\(t_\varepsilon\),\xi_\varepsilon}$, and $t_\eps>0$. As easily checked, if $\phi_{\delta_\varepsilon\(t_\varepsilon\),\xi_\varepsilon}\to 0$ in $H_1^2(M)$ when $\eps\to 0$, then $(u_\varepsilon)$ is of type $u_0-B$. We rewrite equation \eqref{EqEPS} as the couple of equations
\begin{equation}\label{Eq16}
\varPi_{\delta_\varepsilon\(t\),\xi}\(u_0-W_{\delta_\varepsilon\(t\),\xi}+\phi_{\delta_\varepsilon\(t\),\xi}-i^*\(f_\varepsilon\(u_0-W_{\delta_\varepsilon\(t\),\xi}+\phi_{\delta_\varepsilon\(t\),\xi}\)\)\)=0
\end{equation}
and
\begin{equation}\label{Eq17}
\varPi^\perp_{\delta_\varepsilon\(t\),\xi}\(u_0-W_{\delta_\varepsilon\(t\),\xi}+\phi_{\delta_\varepsilon\(t\),\xi}-i^*\(f_\varepsilon\(u_0-W_{\delta_\varepsilon\(t\),\xi}+\phi_{\delta_\varepsilon\(t\),\xi}\)\)\)=0\,.
\end{equation}
We begin with solving equation \eqref{Eq17} in Proposition~\ref{Pr1} below:

\begin{proposition}\label{Pr1} Let $u_0\in C^{2,\theta}(M)$ be a positive nondegenerate solution to \eqref{Th2Eq}. Given two positive real numbers $a<b$, there exists a positive constant $C_{a,b}$ such that for $\varepsilon$ small, for any real number $t$ in $\[a,b\]$, and any point $\xi$ in $M$, there exists a unique function $\phi_{\delta_\varepsilon\(t\),\xi}$ in $K^\perp_{\delta_\varepsilon\(t\),\xi}$ which solves equation \eqref{Eq17} and satisfies
\begin{equation}\label{Pr1Eq1}
\left\|\phi_{\delta_\varepsilon\(t\),\xi}\right\|_h\le C_{a,b}\left\{\begin{aligned} 
&\varepsilon\left|\ln\varepsilon\right|&&\text{if }3\leq n\le6\\
&\varepsilon^{\frac{4}{n-2}}&&\text{if }n\ge7\\
&\varepsilon^{\frac{n+2}{2\(n-2\)}}&&\text{if }n\ge7,\,h\equiv c_n\Scal_g,\text{ and }(M,g)\text{ is loc. conf. flat}.
\end{aligned}\right.
\end{equation}
Moreover,\,$\phi_{\delta_\varepsilon\(t\),\xi}$ is continuously differentiable with respect to $t$ and $\xi$.
\end{proposition}

\noindent The sequel of this section is devoted to the proof of Proposition \ref{Pr1}. For $\varepsilon$ small, for any positive real number $\delta$, and any point $\xi$ in $M$, we let the map $L_{\varepsilon,\delta,\xi}:K^\perp_{\delta,\xi}\to K^\perp_{\delta,\xi}$ be defined by
\begin{equation}\label{Eq20}
L_{\varepsilon,\delta,\xi}\(\phi\):=\varPi^\perp_{\delta,\xi}\(\phi-i^*\(f'_\varepsilon\(u_0-W_{\delta,\xi}\)\phi\)\)
\end{equation}
for all $\phi\in K^\perp_{\delta,\xi}$, where $u_0\in C^{2,\theta}(M)$ is a nondegenerate positive solution to \eqref{Th2Eq}, $W_{\delta,\xi}$ is as in \eqref{Eq9} and $K^\perp_{\delta,\xi}$ is as in \eqref{Eq15}. Clearly, we get that $L_{\varepsilon,\delta,\xi}$ is linear and continuous. In Lemma~\ref{Lem1} below, we prove the invertibility of $L_{\varepsilon,\delta,\xi}$ for $\delta$ and $\varepsilon$ small.

\begin{lemma}\label{Lem1}
Given two positive real numbers $a<b$, for $\varepsilon$ small, the map $L_{\varepsilon,\delta_\varepsilon\(t\),\xi}$ is invertible for all real numbers $t$ in $\[a,b\]$ and all points $\xi$ in $M$, where $\delta_\varepsilon\(t\)=t\varepsilon^{2/\(n-2\)}$ and $L_{\varepsilon,\delta_\varepsilon\(t\),\xi}$ is as in \eqref{Eq20}. Moreover, there exists a positive constant $C_{a,b}$ such that for $\varepsilon$ small, for any real number $t\in \[a,b\]$, any point $\xi\in M$, and any function $\phi\in K^\perp_{\delta_\varepsilon\(t\),\xi}$, there holds
\begin{equation}\label{Lem1Eq1}
\left\|L_{\varepsilon,\delta_\varepsilon\(t\),\xi}\(\phi\)\right\|_h\ge C_{a,b}\left\|\phi\right\|_h\,.
\end{equation}
In particular, the inverse map $L^{-1}_{\varepsilon,\delta_\varepsilon\(t\),\xi}$ is continuous.
\end{lemma}

\proof
We prove \eqref{Lem1Eq1}. We proceed by contradiction. We assume that there exist two sequences of positive real numbers $\(\varepsilon_\alpha\)_\alpha$ and $\(t_\alpha\)_\alpha$ such that $\varepsilon_\alpha\to0$ for all $\alpha\to +\infty$ and $a\le t_\alpha\le b$, a sequence of points $\(\xi_\alpha\)_\alpha$ in $M$, and a sequence of functions $\(\phi_\alpha\)_\alpha$ such that
\begin{equation}\label{Lem1Eq2}
\phi_\alpha\in K^\perp_{\delta_{\varepsilon_\alpha}\(t_\alpha\),\xi_\alpha}\,,\quad\left\|\phi_\alpha\right\|_h=1\,,\quad\text{and }\left\|L_{\varepsilon_\alpha,\delta_{\varepsilon_\alpha}\(t_\alpha\),\xi_\alpha}\(\phi_\alpha\)\right\|_h\longrightarrow0
\end{equation}
as $\alpha\to+\infty$. We define $W_\alpha:= W_{\delta_{\varepsilon_\alpha}\(t_\alpha\),\xi_\alpha}$. First, we claim that
\begin{equation}\label{Lem1Eq3}
\left\|\phi_\alpha-i^*\(f'_{\varepsilon_\alpha}\(u_0-W_{\alpha}\)\phi_\alpha\)\right\|_h\longrightarrow0
\end{equation}
as $\alpha\to+\infty$. Passing if necessary to a subsequence, we may assume that all the points $\xi_\alpha$ belong to a small open subset $\varOmega$ in $M$ on which there exists a smooth orthonormal frame. Thanks to this frame, we identify the tangent space $T_\xi M$ with $\mathbb{R}^n$ for all points $\xi$ in $\varOmega$, so that $\exp_\xi$ is in fact the composition of the standard exponential map with a linear isometry $\varPsi_\xi:\mathbb{R}^n\to T_\xi M$ which is smooth with respect to $\xi$. We define
\begin{equation}\label{Lem1Eq4}
Z_{0,\alpha}:=Z_{\delta_{\varepsilon_\alpha}\(t_{\alpha}\),\xi_\alpha}\quad\text{and}\quad Z_{i,\alpha}=Z_{\delta_{\varepsilon_\alpha}\(t_{\alpha}\),\xi_\alpha,e_i}
\end{equation}
for all $i=1,\dotsc,n$, where $e_i$ is the $i$-th vector in the canonical basis of $\mathbb{R}^n$ and the functions $Z_{\delta_{\varepsilon_\alpha}\(t_{\alpha}\),\xi_\alpha}$ and $Z_{\delta_{\varepsilon_\alpha}\(t_{\alpha}\),\xi_\alpha,e_i}$ are as in \eqref{Eq12}--\eqref{Eq13}. For any $\alpha$, by definition of $L_{\varepsilon,\delta_{\varepsilon_\alpha}\(t_\alpha\),\xi_\alpha}$, we get that
\begin{equation}\label{Lem1Eq5}
\phi_\alpha-i^*\(f'_{\varepsilon_\alpha}\(u_0-W_{\alpha}\)\phi_\alpha\)-L_{\varepsilon_\alpha,\delta_{\varepsilon_\alpha}\(t_\alpha\),\xi_\alpha}\(\phi_\alpha\)=\sum_{i=0}^n\lambda_{i,\alpha}Z_{i,\alpha}
\end{equation}
for some real numbers $\lambda_{i,\alpha}$, where the functions $Z_{i,\alpha}$ are as in \eqref{Lem1Eq4}. Taking into account \eqref{Lem1Eq2} and \eqref{Lem1Eq5}, one sees that in order to get \eqref{Lem1Eq3}, it suffices to prove that $\lambda_{i,\alpha}\to0$ as $\alpha\to+\infty$ for all $i=0,\dotsc,n$. As is easily checked, for any $i,j=0,\dotsc,n$, there holds
\begin{equation}\label{Lem1Eq6}
\<Z_{i,\alpha},Z_{j,\alpha}\>_h\longrightarrow\left\|\nabla V_i\right\|_2^2\delta_{ij}
\end{equation}
as $\alpha\to +\infty$ where the function $V_i$ is as in \eqref{Eq11} and the real numbers $\delta_{ij}$ are the Kronecker symbols. By \eqref{Lem1Eq5}, \eqref{Lem1Eq6}, and since the functions $\phi_\alpha$ and $L_{\varepsilon_\alpha,\delta_{\varepsilon_\alpha}\(t_\alpha\),\xi_\alpha}\(\phi_\alpha\)$ belong to $K^\perp_{\delta_{\varepsilon_\alpha}\(t_\alpha\),\xi_\alpha}$, for any $i=0,\dotsc,n$, we get that
\begin{equation}\label{Lem1Eq7}
\int_Mf'_{\varepsilon_\alpha}\(u_0-W_{\alpha}\)Z_{i,\alpha}\phi_\alpha dv_g=-\lambda_{i,\alpha}\left\|\nabla V_i\right\|_2^2+\o\(\sum_{j=0}^n\left|\lambda_{j,\alpha}\right|\)
\end{equation}
as $\alpha\to+\infty$. As is easily checked, we get that
\begin{align}
&\int_Mf'_{\varepsilon_\alpha}\(u_0-W_{\alpha}\)Z_{i,\alpha}\phi_\alpha dv_g=\int_Mf'_{\varepsilon_\alpha}\(W_{\alpha}\)Z_{i,\alpha}\phi_\alpha dv_g+\o\(1\)\label{Lem1Eq12}
\end{align}
as $\alpha\to+\infty$. We find
\begin{align}
&\frac{\int_Mf'_{\varepsilon_\alpha}\(W_{\alpha}\)Z_{i,\alpha}\phi_\alpha dv_g}{\crit-1-\varepsilon_\alpha}=\delta_{\varepsilon_\alpha}\(t_{\alpha}\)^{\frac{n-2}{2}\varepsilon_\alpha}\int_{\mathbb{R}^n}\chi_\alpha \Lambda_\alpha^{\crit-1-\epsilon_\alpha} U^{2^*-2-\varepsilon_\alpha}V_i\widetilde{\phi}_\alpha dv_{\widetilde{g}_\alpha}\label{Lem1Eq13}
\end{align}
as $\alpha\to+\infty$, where the functions $U$ and $V_i$ are as in \eqref{Eq10} and \eqref{Eq11}, the cutoff function $\chi$ is as in Section \ref{sec:finite} and
\begin{align}
&\chi_\alpha:=\chi\(\delta_{\varepsilon_\alpha}\(t_{\alpha}\)\left|x\right|\)^{2^*-2-\varepsilon_\alpha},\nonumber\\
&\Lambda_\alpha:= \Lambda_{\xi_\alpha}(\hbox{exp}_{\xi_\alpha}(\delta_{\varepsilon_\alpha}\(t_{\alpha}\) x)),\nonumber\\
&\widetilde{\phi}_\alpha\(x\):=\delta_{\varepsilon_\alpha}\(t_{\alpha}\)^{\frac{n-2}{2}}\chi\(\delta_{\varepsilon_\alpha}\(t_{\alpha}\)\left|x\right|\)\phi_\alpha\(\exp_{\xi_\alpha}\(\delta_{\varepsilon_\alpha}\(t_{\alpha}\)x\)\),\label{Lem1Eq14}\\
&\widetilde{g}_\alpha\(x\):=\exp_{\xi_\alpha}^*g\(\delta_{\varepsilon_\alpha}\(t_{\alpha}\)x\)\label{Lem1Eq15}
\end{align}
for any $\alpha\to +\infty$ and $x\in\rn$ small enough. In the definitions above, the exponential map is taken with respect to the metric $g_{\xi_\alpha}$. Since $\(\phi_\alpha\)_\alpha$ is bounded in $H^2_1\(M\)$, we get that $\big(\widetilde{\phi}_\alpha\big)_\alpha$ is bounded in $D^{1,2}\(\mathbb{R}^n\)$. Passing to a subsequence, we may assume that $\big(\widetilde{\phi}_\alpha\big)_\alpha$ converges weakly to some function $\tilde{\phi}$ in $D^{1,2}\(\mathbb{R}^n\)$. Passing to the limit into \eqref{Lem1Eq13} yields
\begin{equation}\label{Lem1Eq16}
\int_Mf'_{\varepsilon_\alpha}\(W_{\alpha}\)Z_{i,\alpha}\phi_\alpha dv_g\longrightarrow\(2^*-1\)\int_{\mathbb{R}^n}U^{2^*-2}V_i\tilde{\phi} dx
\end{equation}
as $\alpha\to+\infty$. Since the function $V_i$ satisfies the equation $\Delta_{\Eucl}V_i=\(2^*-1\)U^{2^*-2}V_i$ in $\mathbb{R}^n$, and since, for any $\alpha$, the function $\phi_\alpha$ belongs to $K^\perp_{\delta_{\varepsilon_\alpha}\(t_\alpha\),\xi_\alpha}$, passing to the limit as $\alpha\to+\infty$ into the equation $\<Z_{i,\alpha},\phi_\alpha\>_h=0$, we get that
\begin{equation}\label{Lem1Eq17}
\int_{\mathbb{R}^n}\<\nabla V_i,\nabla\tilde{\phi}\>dx=\(2^*-1\)\int_{\mathbb{R}^n}U^{2^*-2}V_i\tilde{\phi} dx=0\,.
\end{equation}
By \eqref{Lem1Eq7}, \eqref{Lem1Eq12}, \eqref{Lem1Eq16}, and \eqref{Lem1Eq17}, we get that
$$\lambda_{i,\alpha}=\o\(1\)+\o\(\sum_{j=0}^n\left|\lambda_{j,\alpha}\right|\)$$
as $\alpha\to+\infty$. It follows that $\lambda_{i,\alpha}\to0$ as $\alpha\to+\infty$ for all $i=0,\dotsc,n$. The claim \eqref{Lem1Eq3} then follows from \eqref{Lem1Eq2} and \eqref{Lem1Eq5}. 

\medskip\noindent For any sequence $\(\varphi_\alpha\)_\alpha$ in $H^2_1\(M\)$, and by \eqref{Lem1Eq3}, we get that
\begin{eqnarray}
&&\left|\<\phi_\alpha,\varphi_\alpha\>_h-\int_Mf'_{\varepsilon_\alpha}\(u_0-W_{\alpha}\)\phi_\alpha\varphi_\alpha
dv_g\right|\label{Lem1Eq18}\\
&&\qquad=\left|\<\phi_\alpha-i^*\(f'_{\varepsilon_\alpha}\(u_0-W_{\alpha}\)\phi_\alpha\),\varphi_\alpha\>_h\right|\nonumber\\
&&\qquad\le\left\|\phi_\alpha-i^*\(f'_{\varepsilon_\alpha}\(u_0-W_{\alpha}\)\phi_\alpha\)\right\|_h\left\|\varphi_\alpha\right\|_h=\o\(\left\|\varphi_\alpha\right\|_h\)\nonumber
\end{eqnarray}
as $\alpha\to+\infty$. \par

\medskip\noindent We claim that $\phi_\alpha\rightharpoonup 0$ weakly in $H_1^2(M)$ when $\alpha\to +\infty$. We prove the claim. Since $(\phi_\alpha)$ is bounded in $H_1^2(M)$, up to a subsequence, there exists $\phi\in H_1^2(M)$ such that $(\phi_\alpha)\rightharpoonup \phi$ weakly in $H_1^2(M)$ when $\alpha\to +\infty$. Then for any $\varphi\in H_1^2(M)$, taking $\varphi_\alpha\equiv \varphi$ in \eqref{Lem1Eq18} and letting $\alpha\to +\infty$ yields
$$\langle \phi,\varphi\rangle_h=\int_M (\crit-1)u_0^{\crit-2}\phi\varphi\, dv_g$$
for all $\varphi\in H_1^2(M)$, and then $\Delta_g\phi+h\phi=(\crit-1)u_0^{\crit-2}\phi$, which implies $\phi\equiv 0$ since $u_0$ is nondegenerate. This proves the claim.

\medskip\noindent We claim that $\tilde{\phi}_\alpha\rightharpoonup 0$ weakly in $D_1^2(\rn)$ when $\alpha\to +\infty$, where $\tilde{\phi}_\alpha$ has been defined in \eqref{Lem1Eq14}. We prove the claim. Given a smooth function $\varphi$ with compact support in $\mathbb{R}^n$, we define
$$\varphi_\alpha\(x\):=\chi\(d_{g_{\xi_\alpha}}\(x,\xi_\alpha\)\)\delta_{\varepsilon_\alpha}\(t_\alpha\)^{\frac{2-n}{2}}\varphi\(\delta_{\varepsilon_\alpha}\(t_\alpha\)^{-1}\exp_{\xi_\alpha}^{-1}\(x\)\)\,$$
for all $x\in M$. It follows from \eqref{Lem1Eq18} together with a change of variable that
\begin{eqnarray}
&&\int_{\mathbb{R}^n}\Lambda_\alpha^{-2}\big<\nabla\widetilde{\phi}_\alpha,\nabla\varphi\big>_{\widetilde{g}_\alpha}dv_{\widetilde{g}_\alpha}+\delta_{\varepsilon_\alpha}\(t_\alpha\)^2\int_{\mathbb{R}^n}\Lambda_\alpha^{-\crit}h\(\exp_{\xi_\alpha}\(\delta_{\varepsilon_\alpha}\(t_\alpha\)x\)\)\widetilde{\phi}_\alpha\varphi dv_{\widetilde{g}_\alpha}\label{Lem1Eq19}\\
&&=\delta_{\varepsilon_\alpha}\(t_\alpha\)^2\int_{\mathbb{R}^n}\Lambda_\alpha^{-\crit}f'_{\varepsilon_\alpha}\(u_{0,\alpha}(x)-W_{\alpha}\(\exp_{\xi_\alpha}\(\delta_{\varepsilon_\alpha}\(t_\alpha\)x\)\)\)\widetilde{\phi}_\alpha\varphi dv_{\widetilde{g}_\alpha}+\o\(1\)\nonumber
\end{eqnarray}
as $\alpha\to+\infty$, where $u_{0,\alpha}(\cdot):=u_0\(\exp_{\xi_\alpha}\(\delta_{\varepsilon_\alpha}\(t_\alpha\)\cdot\)\)$, $\widetilde{\phi}_\alpha$ and $\widetilde{g}_\alpha$ are as in \eqref{Lem1Eq14} and \eqref{Lem1Eq15}. One easily checks that 
\begin{equation}
\Lambda_\alpha^{-\crit}\delta_{\varepsilon_\alpha}\(t_\alpha\)^2f'_{\varepsilon_\alpha}\(u_0\(\exp_{\xi_\alpha}\(\delta_{\varepsilon_\alpha}\(t_\alpha\)\cdot\)\)-W_{\alpha}\(\exp_{\xi_\alpha}\(\delta_{\varepsilon_\alpha}\(t_\alpha\)\cdot\)\)\)\nonumber
\end{equation}
goes to $\(2^*-1\)U^{2^*-2}$ as $\alpha\to+\infty$ in $C^0_{loc}\(\mathbb{R}^n\)$. Moreover, since $\big(\widetilde{\phi}_\alpha\big)_\alpha$ converges weakly to $\tilde{\phi}$ in $D_1^{2}\(\mathbb{R}^n\)$, passing to the limit into \eqref{Lem1Eq19} as $\alpha\to+\infty$ yields
\begin{equation}\label{Lem1Eq21}
\int_{\mathbb{R}^n}\<\nabla\tilde{\phi},\nabla\varphi\>dx=\(2^*-1\)\int_{\mathbb{R}^n}U^{2^*-2}\tilde{\phi}\varphi dx\,.
\end{equation}
Since \eqref{Lem1Eq21} holds for all smooth functions $\varphi$ with compact support in $\mathbb{R}^n$, we get that the function $\tilde{\phi}$ satisfies the equation $\Delta_{\Eucl}\tilde{\phi}=\(2^*-1\)U^{2^*-2}\tilde{\phi}$ in $\mathbb{R}^n$. By Bianchi--Egnell~\cite{BiaEgn}, it follows that $\tilde{\phi}=\sum_{i=0}^n\lambda_iV_i$ for some real numbers $\lambda_i$. It then follows from the orthogonality condition \eqref{Lem1Eq17} that $\tilde{\phi}\equiv 0$ is identically zero. This proves the claim.

\medskip\noindent Letting $\varphi_\alpha:=\phi_\alpha$ and using \eqref{Lem1Eq18} together with a change of variable, we get
\begin{align}
\left\|\phi_\alpha\right\|^2_h&=(\crit-1-\eps_\alpha)\int_{M}\left|u_0-W_{\alpha}\right|^{\crit-2-\eps_\alpha}\phi_\alpha^2\, dv_g+\o\(1\)\nonumber\\
&\leq C\int_M\phi_\alpha^2\, dv_g+C\int_M \left|W_{\alpha}\right|^{\crit-2-\eps_\alpha}\phi_\alpha^2\, dv_g+\o\(1\)\nonumber\\
&\leq  C\int_M\phi_\alpha^2\, dv_g+C\int_M |U|^{\crit-2-\eps_\alpha}\tilde\phi_\alpha^2\, dv_{\tilde{g}_\alpha}+\o\(1\)\nonumber
\end{align}
as $\alpha\to+\infty$, where $\widetilde{\phi}_\alpha$ and $\widetilde{g}_\alpha$ are as in \eqref{Lem1Eq14} and \eqref{Lem1Eq15}. Since $\phi_\alpha\to 0$ strongly in $L^2(M)$, $\big(\widetilde{\phi}_\alpha^2\big)_\alpha$ is bounded in $L^{\frac{n}{n-2}}\(\mathbb{R}^n\)$ and converges almost everywhere to $0$, standard elliptic theory yields $\phi_\alpha\to0$ as $\alpha\to+\infty$ in $H^2_1\(M\)$. This is a contradiction with \eqref{Lem1Eq2}. This ends the proof of \eqref{Lem1Eq1}. 

\medskip\noindent The invertibility of $L_{\varepsilon,\delta_\varepsilon\(t\),\xi}$ follows from the Fredholm alternative. This ends the proof of Lemma~\ref{Lem1}.
\endproof
\noindent Now, we prove Proposition~\ref{Pr1} by using Lemma~\ref{Lem1} together with the error estimate in Section \ref{sec:error}.

\proof[Proof of Proposition~\ref{Pr1}]
We let $a$ and $b$ be two positive real numbers such that $a<b$. For $\varepsilon$ small, for any real number $t$ in $\[a,b\]$, and any point $\xi$ in $M$, equation \eqref{Eq17} is equivalent to 
\begin{equation}\label{Pr1Eq2}
L_{\varepsilon,\delta_\varepsilon\(t\),\xi}\(\phi\)=N_{\varepsilon,\delta_\varepsilon\(t\),\xi}\(\phi\)+R_{\varepsilon,\delta_\varepsilon\(t\),\xi}\,,
\end{equation}
where $\delta_\varepsilon\(t\)=t\varepsilon^{2/\(n-2\)}$, $L_{\varepsilon,\delta_\varepsilon\(t\),\xi}$ is as in \eqref{Eq20}, and
\begin{align}
N_{\varepsilon,\delta_\varepsilon\(t\),\xi}\(\phi\)&:=\varPi^\perp_{\delta_\varepsilon\(t\),\xi}\big(i^*
\big(f_\varepsilon\(u_0-W_{\delta_\varepsilon\(t\),\xi}+\phi\)\label{Pr1Eq3}\\
&\qquad-f_\varepsilon\(u_0-W_{\delta_\varepsilon\(t\),\xi}\)-f'_\varepsilon\(u_0-W_{\delta_\varepsilon\(t\),\xi}\)\phi\big)\big)\nonumber
\end{align}
and
\begin{equation}\label{Pr1Eq4}
R_{\varepsilon,\delta_\varepsilon\(t\),\xi}:=\varPi^\perp_{\delta_\varepsilon\(t\),\xi}\(i^*\(f_\varepsilon\left(u_0-W_{\delta_\varepsilon\(t\),\xi}\right)\)-u_0+W_{\delta_\varepsilon\(t\),\xi}\).
\end{equation}
By Lemma~\ref{Lem1}, for $\varepsilon$ small, we get that the map $L_{\varepsilon,\delta_\varepsilon\(t\),\xi}$ is invertible for all real numbers $t$ in $\[a,b\]$ and all points $\xi$ in $M$. We then let the map $T_{\varepsilon,\delta_\varepsilon\(t\),\xi}:K^\perp_{\delta_\varepsilon\(t\),\xi}\to K^\perp_{\delta_\varepsilon\(t\),\xi}$ be defined for all $\phi\in K^\perp_{\delta_\varepsilon\(t\),\xi}$ by
$$T_{\varepsilon,\delta_\varepsilon\(t\),\xi}\(\phi\):=L^{-1}_{\varepsilon,\delta_\varepsilon\(t\),\xi}\(N_{\varepsilon,\delta_\varepsilon\(t\),\xi}\(\phi\)+R_{\varepsilon,\delta_\varepsilon\(t\),\xi}\),$$
where $N_{\varepsilon,\delta_\varepsilon\(t\),\xi}\(\phi\)$ and $R_{\varepsilon,\delta_\varepsilon\(t\),\xi}$ are as in \eqref{Pr1Eq3} and \eqref{Pr1Eq4}. For any positive real number $\varXi$, we let $B_{\varepsilon,\delta_\varepsilon\(t\),\xi}\(\varXi\)$ be the closed ball defined by
$$\overline{B}_{\varepsilon,\delta_\varepsilon\(t\),\xi}\(\varXi\):=\left\{\phi\in K^\perp_{\delta_\varepsilon\(t\),\xi}\, /\, \ \left\|\phi\right\|_h\le\varXi\nu_\varepsilon\right\},$$
where $\nu_\varepsilon>0$ is the error obtained in Lemma \ref{Lem4} of Section \ref{sec:error}, namely
\begin{equation}\label{Pr1Eq5}
\nu_\varepsilon:=\left\{\begin{aligned} 
&\varepsilon\left|\ln\varepsilon\right|&&\text{if }n\le6\\
&\varepsilon^{\frac{4}{n-2}}&&\text{if }n\ge7\\
&\varepsilon^{\frac{n+2}{2\(n-2\)}}&&\text{if }n\ge7,\,h\equiv c_n\Scal_g,\text{ and }(M,g)\text{ loc. conformally flat}.
\end{aligned}\right.
\end{equation}
\medskip\noindent We fix $\theta_0\in (0,\min\{1,\crit-2\})$, so that $u\mapsto f_\eps(u)$ is locally in $C^{1,\theta_0}$ on $H_1^2(M)$ uniformly with respect to $\eps>0$ small. By Lemma~\ref{Lem1} and by continuity of $i^*$,  for $\varepsilon$ small, for any real number $t$ in $\[a,b\]$, any point $\xi$ in $M$, and any functions $\phi$, $\phi_1$, and $\phi_2$ in $H_1^2(M)$, we get that
\begin{equation}\label{Pr1Eq10}
\left\|T_{\varepsilon,\delta_\varepsilon\(t\),\xi}\(\phi_1\)-T_{\varepsilon,\delta_\varepsilon\(t\),\xi}\(\phi_2\)\right\|_h\le C\cdot \left(\max\{\Vert \phi_1\Vert_h^{\theta_0},\Vert \phi_2\Vert_h^{\theta_0}\}\right)\cdot \Vert \phi_1-\phi_2\Vert_h\nonumber
\end{equation}
for some positive constant $C$ independent of $\varXi$, $\varepsilon$, $t$, $\xi$, $\phi$, $\phi_1$, and $\phi_2$, where $\nu_\varepsilon$ is as in \eqref{Pr1Eq5}. By Lemma~\ref{Lem4}, we have that $\left\|T_{\varepsilon,\delta_\varepsilon\(t\),\xi}\(0\)\right\|\leq C\nu_\eps$. We then get that for $\varXi>0$ large enough, and then for $\varepsilon$ small, for any real number $t$ in $\[a,b\]$, and any point $\xi$ in $M$, then the map $T_{\varepsilon,\delta_\varepsilon\(t\),\xi}$ is a contraction map from the closed ball $\overline{B}_{\varepsilon,\delta_\varepsilon\(t\),\xi}\(\varXi\)$ into itself . We then get that the map $T_{\varepsilon,\delta_\varepsilon\(t\),\xi}$ admits a unique fixed point $\phi_{\delta_\varepsilon\(t\),\xi}$ in the ball $\overline{B}_{\varepsilon,\delta_\varepsilon\(t\),\xi}\(\varXi\)$. In other words, the function $\phi_{\delta_\varepsilon\(t\),\xi}$ is the unique solution of equation \eqref{Pr1Eq2}, or equivalently \eqref{Eq17}, which satisfies \eqref{Pr1Eq1} with $C_{a,b}=\varXi$. 

\medskip\noindent The continuous differentiability of $(t,\xi)\mapsto \phi_{\delta_\varepsilon\(t\),\xi}$ on $(a,b)\times M$ is standard. This ends the proof of Proposition~\ref{Pr1}.\endproof

\section{The reduced problem}\label{sec:reduced}
For $\varepsilon$ small, we introduce the functional $J_\varepsilon$ defined on $H^2_1\(M\)$ by
\begin{equation}
J_\varepsilon\(u\):=\frac{1}{2}\int_M\left|\nabla u\right|^2_gdv_g+\frac{1}{2}\int_Mhu^2dv_g-\int_MF_\varepsilon\(u\)dv_g\,,\nonumber
\end{equation}
where $F_\varepsilon\(u\):=\int_0^uf_\varepsilon\(s\)ds$. The critical points of $J_\varepsilon$ are the solutions of equation \eqref{Eq8}. For any positive real number $t$ and any point $\xi$ in $M$, we define
\begin{equation}\label{Eq19}
\mathcal{J}_\varepsilon\(t,\xi\):=J_\varepsilon\(u_0-W_{\delta_\varepsilon\(t\),\xi}+\phi_{\delta_\varepsilon\(t\),\xi}\),
\end{equation}
where $W_{\delta_\varepsilon\(t\),\xi}$ is as in \eqref{Eq9} and $\phi_{\delta_\varepsilon\(t\),\xi}$ is given by Proposition~\ref{Pr1}. We solve equation \eqref{Eq16} in Proposition~\ref{Pr2} below: 

\begin{proposition}\label{Pr2}\samepage Let $u_0\in C^{2,\theta}(M)$ be a positive nondegenerate solution to \eqref{Th2Eq}. Assume that either $\{h\in C^{0,\theta}(M)\hbox{ and }3\le n\le 6\}$ or $\{h\in C^2(M)\hbox{ and }3\le n\le9\}$ or $\{(M,g)$ is locally conformally flat and $h\equiv c_n\Scal_g\}$. Then
\begin{multline}\label{Pr2Eq1}
\mathcal{J}_\varepsilon\(t,\xi\)=c_1(n,u_0)+c_2(n,u_0)\varepsilon+c_3(n)\varepsilon\ln\varepsilon+c_4(n)\varepsilon\ln \frac{1}{t}+c_5(n)\Big(\varepsilon t^{\frac{n-2}{2}}u_0\(\xi\)\\
+\frac{n^{\frac{n-2}{4}}\(n-2\)^{\frac{n-6}{4}}\(n-1\)\omega_{n}\varepsilon^{\frac{4}{n-2}}t^2}{2^{n-1}\(n-4\)\omega_{n-1}}
\cdot\(h\(\xi\)-c_n\Scal_g\(\xi\)\)\mathbf{1}_{n\ge6}\Big)\\
+\o\(\varepsilon\)
\end{multline}
as $\varepsilon\to0$, uniformly with respect to $t$ in compact subsets of $\mathbb{R}_{>0}$ and with respect to the point $\xi$ in $M$, $\omega_n$ (resp. $\omega_{n-1}$) is the volume of the unit $n$-sphere (resp. $\(n-1\)$-sphere), $c_i(n,u_0)$ ($i=1,2$) are positive constants depending only on $n$, $u_0$, and the manifold, $c_i(n)$ ($i=3,4,5$) depend only on $n$, and $c_4(n),c_5(n)>0$. Moreover, given two positive real numbers $a<b$, for $\varepsilon$ small, if $\(t_\varepsilon,\xi_\varepsilon\)\in(a,b)\times M$ is a critical point of $\mathcal{J}_\varepsilon$, then the function $u_0-W_{\delta_\varepsilon\(t_\varepsilon\),\xi_\varepsilon}+\phi_{\delta_\varepsilon\(t_\varepsilon\),\xi_\varepsilon}$ is a solution to equation \eqref{Eq8}, or equivalently \eqref{EqEPS}.
\end{proposition}

\noindent This section is devoted to the proof of Proposition~\ref{Pr2}. We define the optimal Sobolev constant $K_n$ by
\begin{equation}\label{Eq21}
\frac{1}{K_n}:=\inf_{u\in D_1^2(\rn)\setminus\{0\}}\frac{\Vert \nabla u\Vert_2}{\Vert u\Vert_{\crit}}=\sqrt{\frac{n\(n-2\)\omega_n^{2/n}}{4}}\,,
\end{equation}
where $\omega_n$ is the volume of the unit $n$-sphere: see Aubin~\cite{Aub1}, Talenti~\cite{Tal}, Rodemich~\cite{rodemich}. The infimum in \eqref{Eq21} is achieved by the function $U$ defined  in \eqref{Eq10}.

\begin{lemma}\label{Lem2} Let $u_0\in C^2(M)$ be a positive solution to \eqref{Eq9}. 
Assume that either $\{h\in C^{0,\theta}(M)\hbox{ and }3\le n\le6\}$ or $\{h\in C^2(M)\hbox{ and }3\le n\le9\}$ or $\{(M,g)$ is locally conformally flat and $h\equiv c_n \Scal_g\}$. Then
\begin{multline}\label{Lem2Eq1}
J_\varepsilon\(u_0-W_{\delta_\varepsilon\(t\),\xi}\)=\frac{1}{n}\int_Mu_0^{2^*}dv_g+\frac{\varepsilon}{2^*}\int_Mu_0^{2^*}\(\ln u_0-\frac{1}{2^*}\)dv_g\\
+\frac{K_n^{-n}}{n}\bigg(1-\beta_n\varepsilon-\frac{n-2}{2}\varepsilon\ln\varepsilon-\frac{\(n-2\)^2}{4}\varepsilon\ln t+\frac{2^n\omega_{n-1}\varepsilon t^{\frac{n-2}{2}}u_0\(\xi\)}{\(n\(n-2\)\)^{\frac{n-2}{4}}\omega_n}\\
+\frac{2\(n-1\)\varepsilon^{\frac{4}{n-2}}t^2}{\(n-2\)\(n-4\)}\(h\(\xi\)-c_n\Scal_g\(\xi\)\)\mathbf{1}_{n\ge6}+\o\(\varepsilon\)\bigg)
\end{multline}
as $\varepsilon\to0$, uniformly with respect to $t$ in compact subsets of $\mathbb{R}_{>0}$ and with respect to the point $\xi$ in $M$, where $\omega_n$ (resp. $\omega_{n-1}$) is the volume of the unit $n$-sphere (resp. $\(n-1\)$-sphere), $K_n$ is as in \eqref{Eq21}, and
\begin{equation}\label{Lem2Eq2}
\beta_n=2^{n-3}\(n-2\)^2\frac{\omega_{n-1}}{\omega_n}\int_0^{+\infty}\frac{r^{\frac{n-2}{2}}\ln\(1+r\)}{\(1+r\)^n}dr+\frac{\(n-2\)^2}{4n}\(1-n\ln\sqrt{n\(n-2\)}\).
\end{equation}
\end{lemma}
\proof
All our estimates in this proof are uniform with respect to $t$ in compact subsets of $\mathbb{R}_{>0}$, with respect to the point $\xi$ in $M$, and with respect to $\varepsilon$ in $\(0,\varepsilon_0\)$ for some fixed positive real number $\varepsilon_0$. Expanding $J_\varepsilon\(u_0-W_{\delta_\varepsilon\(t\),\xi}\)$, using that $u_0$ is a solution to \eqref{Th2Eq} and rough estimates yield
\begin{multline}\label{Lem2Eq3BIS}
J_\varepsilon\(u_0-W_{\delta_\varepsilon\(t\),\xi}\)=\frac{1}{n}\int_Mu_0^{2^*}dv_g+\frac{\varepsilon}{2^*}\int_Mu_0^{2^*}\(\ln
u_0-\frac{1}{2^*}\)dv_g\\
+I_{1,\eps,t,\xi}+I_{2,\eps,t,\xi}-I_{3,\eps,t,\xi}+\O\(\varepsilon^2\)
\end{multline}
when $\eps\to 0$ where 
$$I_{1,\eps,t,\xi}:=\frac{1}{2}\int_M\left|\nabla W_{\delta_\varepsilon\(t\),\xi}\right|_g^2dv_g+\frac{1}{2}\int_MhW_{\delta_\varepsilon\(t\),\xi}^2dv_g-\frac{1}{2^*-\varepsilon}\int_MW_{\delta_\varepsilon\(t\),\xi}^{2^*-\varepsilon}dv_g\,,$$
$$I_{2,\eps,t,\xi}:=\int_Mu_0W_{\delta_\varepsilon\(t\),\xi}^{2^*-1-\varepsilon}dv_g\,,$$
\begin{multline*}
I_{3,\eps,t,\xi}:=\int_M\big(F_\varepsilon\(u_0-W_{\delta_\varepsilon\(t\),\xi}\)-F_\varepsilon\(u_0\)-F_\varepsilon\(W_{\delta_\varepsilon\(t\),\xi}\)\\
+f_\varepsilon\(u_0\)W_{\delta_\varepsilon\(t\),\xi}+f_\varepsilon\(W_{\delta_\varepsilon\(t\),\xi}\)u_0\big)dv_g\,.
\end{multline*}
We estimate these terms separately.

\medskip\noindent{\it Step 1: Estimate of $I_{1,\eps,t,\xi}$ in the locally conformally case when $h\equiv c_n\Scal_g$.}\\
\noindent In case $h\equiv c_n\Scal_g$ and the manifold is locally conformally flat, the conformal change of metric $g_\xi=\varLambda_\xi^{4/\(n-2\)}g$ yields
\begin{multline*}
\frac{1}{2}\int_M\left|\nabla W_{\delta_\varepsilon\(t\),\xi}\right|_g^2dv_g+\frac{1}{2}\int_MhW_{\delta_\varepsilon\(t\),\xi}^2dv_g-\frac{1}{2^*-\varepsilon}\int_MW_{\delta_\varepsilon\(t\),\xi}^{2^*-\varepsilon}dv_g\\
=\frac{1}{2}\int_M\big|\nabla\widetilde{W}_{\delta_\varepsilon\(t\),\xi}\big|_{g_\xi}^2dv_{g_\xi}-\frac{1}{2^*-\varepsilon}\int_M\varLambda_\xi^{-\varepsilon}\widetilde{W}_{\delta_\varepsilon\(t\),\xi}^{2^*-\varepsilon}dv_{g_\xi}\,, 
\end{multline*}
where $\widetilde{W}_{\delta_\varepsilon\(t\),\xi}=W_{\delta_\varepsilon\(t\),\xi}/\varLambda_\xi$. In this case, since the metric $g_\xi$ is flat in $B_\xi\(r_0\)$, we find
\begin{align*}
\int_M\big|\nabla\widetilde{W}_{\delta_\varepsilon\(t\),\xi}\big|_{g_\xi}^2dv_{g_\xi}&=\int_{\rn}|\nabla U|^2\, dx+\O\(\delta_\varepsilon\(t\)^{n-2}\)=K_n^{-n}+\O\(\varepsilon^2\) 
\end{align*}
when $\eps\to 0$. Moreover, since $g_\xi$ is flat around $\xi$, we get that
\begin{align*}
&\frac{1}{2^*-\varepsilon}\int_M\varLambda_\xi^{-\varepsilon}\widetilde{W}_{\delta_\varepsilon\(t\),\xi}^{2^*-\varepsilon}dv_{g_\xi}\nonumber\\
&\quad=\frac{\(n\(n-2\)\)^{\frac{n-2}{4}\(2^*-\varepsilon\)}}{2^*-\varepsilon}\omega_{n-1}\delta_\varepsilon\(t\)^{\frac{n-2}{2}\varepsilon}\int_0^{\frac{r_0}{2\delta_\varepsilon\(t\)}}\frac{r^{n-1}dr}{\(1+r^2\)^{\frac{n-2}{2}\(2^*-\varepsilon\)}}+\O\(\delta_\varepsilon\(t\)^n\)\nonumber\\
&\quad+\O\(\eps \delta_\varepsilon\(t\)\)\nonumber\\
&\quad=\frac{n-2}{2n}K_n^{-n}\bigg(1+\frac{2\beta_n}{n-2}\varepsilon+\varepsilon\ln\varepsilon+\frac{n-2}{2}\varepsilon\ln t\bigg)+\O\(\varepsilon^2\left|\ln\varepsilon\right|^2\)+\O\(\eps \delta_\varepsilon\(t\)\), 
\end{align*}
where $K_n$ is as in \eqref{Eq21}, and $\beta_n$ is as in \eqref{Lem2Eq2}. Therefore, we get that
\begin{equation}\label{I:1:eps:lcf}
I_{1,\eps,t,\xi}=\frac{K_n^{-n}}{n}\bigg(1-\beta_n\varepsilon-\frac{n-2}{2}\varepsilon\ln\varepsilon-\frac{\(n-2\)^2}{4}\varepsilon\ln t\bigg)+o(\eps)\\
\end{equation}
when $\eps\to 0$ uniformly for all $\xi\in M$ and $t$ in a compact of $(0,+\infty)$ when $h\equiv c_n\Scal_g$ and $(M,g)$ is locally conformally flat.

\medskip\noindent{\it Step 2: Estimate of $I_{1,\eps,t,\xi}$ in the general case.}\\
\noindent Cartan's expansion of the metric in geodesic normal coordinates yields for any $\alpha,\beta=1,\dotsc,n$ and for $y$ close to $0$, there holds
\begin{equation}\label{Lem2Eq12}
\sqrt{\left|g\(\exp_\xi y\)\right|}=1-\frac{1}{6}\sum_{\mu,\nu=1}^nR_{\mu\nu}\(\xi\)y^\mu y^\nu+P_3(y)+\O\(\left|y\right|^4\),
\end{equation}
where the function $\left|g\right|$ is the determinant of the metric, the functions $R_{\mu\nu}$ are the components of the Ricci curvature tensor in geodesic normal coordinates associated with the map $\exp_\xi$ and $P_3(y)$ is a homogenous polynomial of degree three. Using \eqref{Lem2Eq12} together with expression of the gradient of a radially symmetrical function in geodesic normal chart, we get that
\begin{eqnarray}
&&\int_M\left|\nabla
W_{\delta_\varepsilon\(t\),\xi}\right|_g^2dv_g=n^{\frac{n-2}{2}}\(n-2\)^{\frac{n+2}{2}}\omega_{n-1}\label{Lem2Eq13}\\
&&\quad\times\int_0^{\frac{r_0}{2\delta_\varepsilon\(t\)}}\frac{r^{n+1}}{\(1+r^2\)^n}\(1-\frac{1}{6n}\Scal_g\(\xi\)\delta_\varepsilon\(t\)^2r^2+\O\(\delta_\varepsilon\(t\)^4r^4\)\)dr\nonumber\\
&&\quad+\O\(\delta_\varepsilon\(t\)^{n-2}\)\nonumber\\
&&=\left\{\begin{aligned}
&K_3^{-3}+\O\(\varepsilon^2\)&&\text{if }n=3\\
&K_4^{-4}\(1+\frac{1}{4}\Scal_g\(\xi\)t^2\varepsilon^2\ln\varepsilon\)+\O\(\varepsilon^2\)&&\text{if
}n=4\\
&K_n^{-n}\(1-\frac{n+2}{6n\(n-4\)}\Scal_g\(\xi\)t^2\varepsilon^{\frac{4}{n-2}}\)+\O\(\varepsilon^{\frac{8}{n-2}}+\varepsilon^2\left|\ln\varepsilon\right|\)&&\text{if
}n\ge5
\end{aligned}\right.\nonumber
\end{eqnarray}
when $\eps\to 0$. Taylor's expansion at $\xi$ yields on the one hand
\begin{eqnarray}
&&\frac{1}{2^*-\varepsilon}\int_MW_{\delta_\varepsilon\(t\),\xi}^{2^*-\varepsilon}dv_g=\frac{\(n\(n-2\)\)^{\frac{n-2}{4}\(2^*-\varepsilon\)}}{2^*-\varepsilon}\omega_{n-1}\delta_\varepsilon\(t\)^{\frac{n-2}{2}\varepsilon}\label{Lem2Eq15}\\
&&\quad\times\int_0^{\frac{r_0}{2\delta_\varepsilon\(t\)}}\frac{r^{n-1}}{\(1+r^2\)^{\frac{n-2}{2}\(2^*-\varepsilon\)}}\bigg(1-\frac{1}{6n}\Scal_g\(\xi\)\delta_\varepsilon\(t\)^2r^2\nonumber\\
&&\quad+\O\(\delta_\varepsilon\(t\)^4r^4\)\bigg)dr+\O\(\delta_\varepsilon\(t\)^n\)\nonumber\\
&&=\frac{n-2}{2n}K_n^{-n}\bigg(1+\frac{2\beta_n}{n-2}\varepsilon+\varepsilon\ln\varepsilon+\frac{n-2}{2}\varepsilon\ln
t-\frac{1}{6\(n-2\)}\Scal_g\(\xi\)t^2\varepsilon^{\frac{4}{n-2}}\bigg)\nonumber\\
&&\quad+\O\(\varepsilon^{\frac{8}{n-2}}+\varepsilon^2\left|\ln\varepsilon\right|^2\)\nonumber
\end{eqnarray}
when $\eps\to 0$. On the other hand,
\begin{align}
&\int_MhW_{\delta_\varepsilon\(t\),\xi}^2dv_g=n^{\frac{n-2}{2}}\(n-2\)^{\frac{n-2}{2}}\delta_\varepsilon\(t\)^2\label{Lem2Eq14}\\
&\qquad\times\int_{B_{\frac{r_0}{2\delta_\varepsilon\(t\)}}(0)}\frac{h(\hbox{exp}_\xi(\delta_\eps(t)x))}{\(1+|x|^2\)^{n-2}}(1+\O(\delta_\varepsilon\(t\)^2|x|^2))\,dx+\O\(\delta_\varepsilon\(t\)^{n-2}\)\nonumber\allowdisplaybreaks\\
&\quad=\left\{\begin{aligned}
&\O\(\varepsilon^2\)&&\text{if }n=3\\
&-\frac{3}{2}K_4^{-4}h\(\xi\)t^2\varepsilon^2\ln\varepsilon+\O\(\varepsilon^2\)&&\text{if }n=4\\
&\frac{4\(n-1\)}{n\(n-2\)\(n-4\)}K_n^{-n}h\(\xi\)t^2\varepsilon^{\frac{4}{n-2}}+\O\(R_\eps\)&&\text{if }n\ge5
\end{aligned}\right.\nonumber
\end{align}
when $\eps\to 0$, where
$$R_\eps:=\left\{\begin{array}{ll}
\varepsilon^{\frac{8}{n-2}}+\varepsilon^2\left|\ln\varepsilon\right|&\hbox{ if }h\in C^2(M)\\
\varepsilon^{\frac{2(2+\theta)}{n-2}}&\hbox{ if }h\in C^{0,\theta}(M).
\end{array}\right.$$
Plugging together \eqref{Lem2Eq13}, \eqref{Lem2Eq14}, and \eqref{Lem2Eq15} yields
\begin{multline}\label{est:I:2}
I_{1,\eps,t,\xi}=\frac{K_n^{-n}}{n}\bigg(1-\beta_n\varepsilon-\frac{n-2}{2}\varepsilon\ln\varepsilon-\frac{\(n-2\)^2}{4}\varepsilon\ln t \\
+\frac{2\(n-1\)\varepsilon^{\frac{4}{n-2}}t^2}{\(n-2\)\(n-4\)}\(h\(\xi\)-c_n\Scal_g\(\xi\)\)\mathbf{1}_{n\ge6}\bigg)+\o\(\varepsilon\)+\O\(R_\eps\)
\end{multline}
when $\eps\to 0$.

\medskip\noindent{\it Step 3: Estimate of $I_{2,\eps,t,\xi}$.}
\noindent\begin{align}
&\int_Mu_0W_{\delta_\varepsilon\(t\),\xi}^{2^*-1-\varepsilon}dv_g=\(n\(n-2\)\)^{\frac{n-2}{4}\(2^*-1-\varepsilon\)}\omega_{n-1}u_0\(\xi\)\delta_\varepsilon\(t\)^{\frac{n-2}{2}\(1+\varepsilon\)}\label{Lem2Eq16}
\\
&\qquad\times\int_0^{\frac{r_0}{2\delta_\varepsilon\(t\)}}\frac{r^{n-1}}{\(1+r^2\)^{\frac{n-2}{2}\(2^*-1-\varepsilon\)}}\bigg(1+\O\(\delta_\varepsilon\(t\)^2r^2\)\bigg)dr+\O\(\delta_\varepsilon\(t\)^{\frac{n+2}{2}}\)\nonumber\\
&\quad=\frac{2^n\omega_{n-1}K_n^{-n}u_0\(\xi\)t^{\frac{n-2}{2}}\varepsilon}{n^{\frac{n+2}{4}}\(n-2\)^{\frac{n-2}{4}}\omega_n}+\O\(\varepsilon^{\frac{n+2}{n-2}}\left|\ln\varepsilon\right|+\varepsilon^2\left|\ln\varepsilon\right|\)\nonumber
\end{align}
when $\eps\to 0$, where $K_n$ is as in \eqref{Eq21}, and $\beta_n$ is as in \eqref{Lem2Eq2}. We have used here that $\int_{\rn}U^{\crit-1}\, dx=\lim_{R\to +\infty}\int_{B_R(0)}\Delta_{\Eucl}U\, dx$ and integrated by parts.

\medskip\noindent{\it Step 4: Estimate of $I_{3,\eps,t,\xi}$.}\\
\noindent We have that
\begin{eqnarray}\label{Lem2Eq17}
&&\bigg|\int_M\left(F_\varepsilon\(u_0-W_{\delta_\varepsilon\(t\),\xi}\)-F_\varepsilon\(u_0\)-F_\varepsilon\(W_{\delta_\varepsilon\(t\),\xi}\)\right.\\
&&\quad\left.
+f_\varepsilon\(u_0\)W_{\delta_\varepsilon\(t\),\xi}+f_\varepsilon\(W_{\delta_\varepsilon\(t\),\xi}\)u_0\)dv_g\bigg|\nonumber\\
&&\le\int_{B_\xi(\sqrt{\delta_\varepsilon\(t\)})}\left|F_\varepsilon\(u_0-W_{\delta_\varepsilon\(t\),\xi}\)-F_\varepsilon\(W_{\delta_\varepsilon\(t\),\xi}\)
+f_\varepsilon\(W_{\delta_\varepsilon\(t\),\xi}\)u_0\right|dv_g\nonumber\\
&&\quad+\int_{M\backslash
B_\xi(\sqrt{\delta_\varepsilon\(t\)})}\left|F_\varepsilon\(u_0-W_{\delta_\varepsilon\(t\),\xi}\)-F_\varepsilon\(u_0\)
+f_\varepsilon\(u_0\)W_{\delta_\varepsilon\(t\),\xi}\right|dv_g\nonumber\\
&&\quad+\int_{B_\xi(\sqrt{\delta_\varepsilon\(t\)})}\left|F_\varepsilon\(u_0\)\right|
dv_g+\int_{B_\xi(\sqrt{\delta_\varepsilon\(t\)})}\left|f_\varepsilon\(u_0\)W_{\delta_\varepsilon\(t\),\xi}\right|dv_g\nonumber\\
&&\quad+\int_{M\backslash
B_\xi(\sqrt{\delta_\varepsilon\(t\)})}\left|F_\varepsilon\(W_{\delta_\varepsilon\(t\),\xi}\)\right|dv_g+\int_{M\backslash
B_\xi(\sqrt{\delta_\varepsilon\(t\)})}\left|f_\varepsilon\(W_{\delta_\varepsilon\(t\),\xi}\)u_0\right|dv_g\,.\nonumber
\end{eqnarray}
As is easily checked, Taylor expansions of $F(u_0-W_{\delta_\varepsilon\(t\),\xi})$ yield
\begin{align}
&\int_{B_\xi(\sqrt{\delta_\varepsilon\(t\)})}\left|F_\varepsilon\(u_0-W_{\delta_\varepsilon\(t\),\xi}\)-F_\varepsilon\(W_{\delta_\varepsilon\(t\),\xi}\) +f_\varepsilon\(W_{\delta_\varepsilon\(t\),\xi}\)u_0\right|dv_g\nonumber\\
&\qquad=\O\(\int_{B_\xi(\sqrt{\delta_\varepsilon\(t\)})}u_0^2W_{\delta_\varepsilon\(t\),\xi}^{2^*-2-\varepsilon}dv_g\),\label{Lem2Eq18}\\
&\int_{M\backslash B_\xi(\sqrt{\delta_\varepsilon\(t\)})}\left|F_\varepsilon\(u_0-W_{\delta_\varepsilon\(t\),\xi}\)-F_\varepsilon\(u_0\) +f_\varepsilon\(u_0\)W_{\delta_\varepsilon\(t\),\xi}\right|dv_g\nonumber\\
&\qquad=\O\(\int_{M\backslash B_\xi(\sqrt{\delta_\varepsilon\(t\)})}u_0^{2^*-2-\varepsilon}W_{\delta_\varepsilon\(t\),\xi}^2dv_g\).\label{Lem2Eq19}
\end{align}
Bounding $u_0$ and $W_{\delta_\varepsilon\(t\),\xi}$ pointwisely roughly from above in \eqref{Lem2Eq18} and \eqref{Lem2Eq19} and plugging this in \eqref{Lem2Eq17} yields 
\begin{equation}\label{Lem2Eq26}
I_{3,\eps,t,\xi}=\left\{\begin{aligned}
&\O\(\varepsilon^2\)&&\text{if }n=3\\
&\O\(\varepsilon^2\left|\ln\varepsilon\right|\)&&\text{if }n=4\\
&\O\(\varepsilon^{\frac{n}{n-2}}\)&&\text{if }n\ge5
\end{aligned}\right.
\end{equation}
when $\eps\to 0$.

\medskip\noindent{\it Step 5: End of proof of Lemma \ref{Lem2}.}\\
\noindent  The asymptotic expansion \eqref{Lem2Eq1} follows from \eqref{Lem2Eq3BIS}, \eqref{I:1:eps:lcf}, \eqref{est:I:2}, \eqref{Lem2Eq16} and \eqref{Lem2Eq26}.
\endproof
In Lemma~\ref{Lem3} below, we show that the first order terms in the asymptotic expansion of $\mathcal{J}_\varepsilon\(t,\xi\)$, defined in \eqref{Eq19}, are the same as for $J_\varepsilon\(u_0-W_{\delta_\varepsilon\(t\),\xi}\)$.

\begin{lemma}\label{Lem3}
Assume that either $\{3\le n\le9\}$ or $\{(M,g)$ is locally conformally flat and $h\equiv c_n\Scal_g\}$. Then
\begin{equation}\label{Lem3Eq1}
\mathcal{J}_\varepsilon\(t,\xi\)=J_\varepsilon\(u_0-W_{\delta_\varepsilon\(t\),\xi}\)+\o\(\varepsilon\)
\end{equation}
as $\varepsilon\to0$, uniformly with respect to $t$ in compact subsets of $\mathbb{R}_{>0}$ and with respect to the point $\xi$ in $M$.
\end{lemma}

\proof
All the estimates in this proof are uniform with respect to $t$ on compact subsets of $\mathbb{R}_{>0}$, with respect to the point $\xi$ in $M$, and with respect to $\varepsilon$ in $\(0,\varepsilon_0\)$ for some fixed positive real number $\varepsilon_0$. We get that
\begin{multline}\label{Lem3Eq2}
\mathcal{J}_\varepsilon\(t,\xi\)-J_\varepsilon\(u_0-W_{\delta_\varepsilon\(t\),\xi}\)\\
\quad=\<u_0-W_{\delta_\varepsilon\(t\),\xi}-i^*\(f_\varepsilon\(u_0-W_{\delta_\varepsilon\(t\),\xi}\)\),\phi_{\delta_\varepsilon\(t\),\xi}\>_h+\O\(\left\|\phi_{\delta_\varepsilon\(t\),\xi}\right\|_h^2\)
\end{multline}
when $\eps\to 0$. Proposition~\ref{Pr1} and Lemma~\ref{Lem4} yield
\begin{multline}
\<u_0-W_{\delta_\varepsilon\(t\),\xi}-i^*\(f_\varepsilon\(u_0-W_{\delta_\varepsilon\(t\),\xi}\)\),\phi_{\delta_\varepsilon\(t\),\xi}\>_h+\O\(\left\|\phi_{\delta_\varepsilon\(t\),\xi}\right\|^2_h\)\label{Lem3Eq3}\\
\qquad=\left\{\begin{aligned}
&\O\(\varepsilon^2\left|\ln\varepsilon\right|^2\)&&\text{if }n\le6\\
&\O\(\varepsilon^{\frac{8}{n-2}}\)&&\text{if }n\ge7\\
&\O\(\varepsilon^{\frac{n+2}{n-2}}\)&&\text{if }n\ge7,\,h\equiv c_n\Scal_g,\;(M,g)\text{ loc. conformally flat}
\end{aligned}\right.
\end{multline}
\noindent when $\eps\to 0$. Finally, \eqref{Lem3Eq1} follows from \eqref{Lem3Eq2}, and \eqref{Lem3Eq3}.
\endproof
The asymptotic expansion \eqref{Pr2Eq1} follows from Lemmas~\ref{Lem2} and~\ref{Lem3}. Now, we prove the second part of Proposition~\ref{Pr2}.

\proof[End of proof of Proposition~\ref{Pr2}]
Given two positive real numbers $a<b$, it remains to prove that for $\varepsilon$ small, if $\(t_\varepsilon,\xi_\varepsilon\)\in\[a,b\]\times M$ is a critical point of $\mathcal{J}_\varepsilon$, then the function $u_0-W_{\delta_\varepsilon\(t_\varepsilon\),\xi_\varepsilon}+\phi_{\delta_\varepsilon\(t_\varepsilon\),\xi_\varepsilon}$ is a solution of equation \eqref{Eq8}. In order to prove this claim, we consider a sequence of points $\(\xi_\alpha\)_\alpha$ in $M$ and two sequences of positive real numbers $\(\varepsilon_\alpha\)_\alpha$ and $\(t_\alpha\)_\alpha$ such that $\varepsilon_\alpha\to0$ as $\alpha\to +\infty$, $a\le t_\alpha\le b$, and $\(t_\alpha,\xi_\alpha\)$ is a critical point of $\mathcal{J}_{\varepsilon_\alpha}$ for all $\alpha$. It is enough to show that for $\alpha$ large, the function $u_0-W_{\delta_{\varepsilon_\alpha}\(t_\alpha\),\xi_\alpha}+\phi_{\delta_{\varepsilon_\alpha}\(t_\alpha\),\xi_\alpha}$ is a solution of equation \eqref{Eq8}. As in the proof of Lemma \ref{Lem1}, up to a subsequence, we identify the tangent space with $\mathbb{R}^n$ around the $\xi_\alpha$'s. We define
\begin{equation}\label{Pr2Eq2}
Z_{0,\delta_{\varepsilon_\alpha}\(t_{\alpha}\),\xi_\alpha}:=Z_{\delta_{\varepsilon_\alpha}\(t_{\alpha}\),\xi_\alpha}\quad\text{and}\quad Z_{i,\delta_{\varepsilon_\alpha}\(t_{\alpha}\),\xi_\alpha}:=Z_{\delta_{\varepsilon_\alpha}\(t_{\alpha}\),\xi_\alpha,e_i}
\end{equation}
for all $i=1,\dotsc,n$, where $e_i$ is the $i$-th vector in the canonical basis of $\mathbb{R}^n$ and the functions $Z_{\delta_{\varepsilon_\alpha}\(t_{\alpha}\),\xi_\alpha}$ and $Z_{\delta_{\varepsilon_\alpha}\(t_{\alpha}\),\xi_\alpha,e_i}$ are as in \eqref{Eq12} and \eqref{Eq13}. By Proposition~\ref{Pr1}, we get that
\begin{equation}\label{Pr2Eq3}
DJ_{\varepsilon_\alpha}\(u_0-W_{\delta_{\varepsilon_\alpha}\(t_{\alpha}\),\xi_\alpha}+\phi_{\delta_{\varepsilon_\alpha}\(t_{\alpha}\),\xi_\alpha}\)=\sum_{i=0}^n\lambda_{i,\alpha}\<Z_{i,\delta_{\varepsilon_\alpha}\(t_{\alpha}\),\xi_\alpha},\cdot\>_h
\end{equation}
for some real numbers $\lambda_{i,\alpha}$, where the functions $Z_{i,\delta_{\varepsilon_\alpha}\(t_{\alpha}\),\xi_\alpha}$ are as in \eqref{Pr2Eq2}. It follows from \eqref{Pr2Eq3} that
\begin{equation}\label{Pr2Eq4}
\frac{\partial\mathcal{J}_{\varepsilon_\alpha}}{\partial t}\(t_\alpha,\xi_\alpha\)=\sum_{i=0}^n\lambda_{i,\alpha}\langle Z_{i,\delta_{\varepsilon_\alpha}\(t_{\alpha}\),\xi_\alpha},\left.\frac{d}{dt}\(-W_{\delta_{\varepsilon_\alpha}\(t\),\xi_\alpha}+\phi_{\delta_{\varepsilon_\alpha}\(t\),\xi_\alpha}\)\right|_{t=t_\alpha}\rangle_h\,.
\end{equation}
On the one hand, we find
\begin{equation}\label{Pr2Eq5}
\left.\frac{d}{dt}W_{\delta_{\varepsilon_\alpha}\(t\),\xi_\alpha}\right|_{t=t_\alpha}=\frac{n^{\frac{n-2}{4}}\(n-2\)^{\frac{n+2}{4}}}{2t_\alpha}Z_{0,\delta_{\varepsilon_\alpha}\(t_{\alpha}\),\xi_\alpha}\,.
\end{equation}
On the other hand, for any $i=0,\dotsc,n$ and any $\alpha$, since the function $\phi_{\delta_{\varepsilon_\alpha}\(t_{\alpha}\),\xi_\alpha}$ belongs to $K^\perp_{\delta_{\varepsilon_\alpha}\(t_{\alpha}\),\xi_\alpha}$, differentiating $\<Z_{i,\delta_{\varepsilon_\alpha}\(t\),\xi_\alpha},\phi_{\delta_{\varepsilon_\alpha}\(t\),\xi_\alpha}\>_h=0$ with respect to $t$ yields
\begin{equation}\label{Pr2Eq6}
\<Z_{i,\delta_{\varepsilon_\alpha}\(t_{\alpha}\),\xi_\alpha},\left.\frac{d}{dt}\phi_{\delta_{\varepsilon_\alpha}\(t\),\xi_\alpha}\right|_{t=t_\alpha}\>_h=-\<\left.\frac{d}{dt}Z_{i,\delta_{\varepsilon_\alpha}\(t\),\xi_\alpha}\right|_{t=t_\alpha},\phi_{\delta_{\varepsilon_\alpha}\(t_{\alpha}\),\xi_\alpha}\>_h.
\end{equation}
Moreover, one easily checks
\begin{equation}\label{Pr2Eq7}
\left\|\left.\frac{d}{dt}Z_{i,\delta_{\varepsilon_\alpha}\(t\),\xi_\alpha}\right|_{t=t_\alpha}\right\|_h=\O\(1\)
\end{equation}
as $\alpha\to+\infty$. Proposition~\ref{Pr1}, \eqref{Pr2Eq6}, \eqref{Pr2Eq7},
\eqref{Lem1Eq6}, \eqref{Pr2Eq4}, and \eqref{Pr2Eq5} yield 
\begin{equation}\label{Pr2Eq10}
\frac{\partial\mathcal{J}_{\varepsilon_\alpha}}{\partial t}\(t_\alpha,\xi_\alpha\)=-\frac{n^{\frac{n-2}{4}}\(n-2\)^{\frac{n+2}{4}}}{2t_\alpha}\lambda_{0,\alpha}\left\|\nabla V_0\right\|_2^2+\o\(\sum_{i=0}^n\left|\lambda_{i,\alpha}\right|\)
\end{equation}
as $\alpha\to+\infty$, where the function $V_0$ is as in \eqref{Eq11}. For any $i=1,\dotsc,n$, by \eqref{Pr2Eq3}, we get that
\begin{multline*}
\left.\frac{d}{dy_i}\mathcal{J}_{\varepsilon_\alpha}\(t_\alpha,\exp_{\xi_\alpha}y\)\right|_{y=0}\\
=\sum_{j=0}^n\lambda_{j,\alpha}\<Z_{j,\delta_{\varepsilon_\alpha}\(t_{\alpha}\),\xi_\alpha},\left.\frac{d}{dy_i}\(-W_{\delta_{\varepsilon_\alpha}\(t_\alpha\),\exp_{\xi_\alpha}y}+\phi_{\delta_{\varepsilon_\alpha}\(t_\alpha\),\exp_{\xi_\alpha}y}\)\right|_{y=0}\>_h,
\end{multline*}
where the exponential map is taken with respect to the metric $g_{\xi_\alpha}$. On the one hand, direct computations yield
\begin{equation*}
\left.\frac{d}{dy_i}W_{\delta_{\varepsilon_\alpha}\(t_\alpha\),\exp_{\xi_\alpha}y}\right|_{y=0}=\frac{n^{\frac{n-2}{4}}\(n-2\)^{\frac{n+2}{4}}}{\delta_{\varepsilon_\alpha}\(t\)}\(Z_{i,\delta_{\varepsilon_\alpha}\(t_{\alpha}\),\xi_\alpha}+R_{i,\delta_{\varepsilon_\alpha}\(t_{\alpha}\),\xi_\alpha}\),
\end{equation*}
where $R_{i,\delta_{\varepsilon_\alpha}\(t_{\alpha}\),\xi_\alpha}\to0$ as $\alpha\to+\infty$ in $H_1^2\(M\)$. For any $i=1,\dotsc,n$, $j=0,\dotsc,n$, and any $\alpha$, since the function $\phi_{\delta_{\varepsilon_\alpha}\(t_{\alpha}\),\xi_\alpha}\in K^\perp_{\delta_{\varepsilon_\alpha}\(t_{\alpha}\),\xi_\alpha}$, differentiating the equation $\big<Z_{j,\delta_{\varepsilon_\alpha}\(t_\alpha\),\exp_{\xi_\alpha}y},\phi_{\delta_{\varepsilon_\alpha}\(t_\alpha\),\exp_{\xi_\alpha}y}\big>_h=0$ with respect to $y_i$ at $0$ yields
\begin{equation*}
\left\langle Z_{j,\delta_{\varepsilon_\alpha}\(t_{\alpha}\),\xi_\alpha},\frac{d}{dy_i}\phi_{\delta_{\varepsilon_\alpha}\(t_\alpha\),\exp_{\xi_\alpha}y }\right\rangle_h=-\left\langle\frac{d}{dy_i}Z_{j,\delta_{\varepsilon_\alpha}\(t_\alpha\),\exp_{\xi_\alpha}y},\phi_{\delta_{\varepsilon_\alpha}\(t_{\alpha}\),\xi_\alpha}\right\rangle_h.
\end{equation*}
Moreover, one easily checks
\begin{equation*}
\left\|\left.\frac{d}{dy_i}Z_{j,\delta_{\varepsilon_\alpha}\(t_\alpha\),\exp_{\xi_\alpha}y}\right|_{y=0}\right\|_h=\O\(\frac{1}{\delta_{\varepsilon_\alpha}\(t\)}\)
\end{equation*}
as $\alpha\to+\infty$. Similarly to the derivative in the $t-$direction, we get that
\begin{equation}\label{Pr2Eq17}
\delta_{\varepsilon_\alpha}\(t_\alpha\)\left.\frac{d}{dy_i}\mathcal{J}_{\varepsilon_\alpha}\(t_\alpha,\exp_{\xi_\alpha}y\)\right|_{y=0}=-n^{\frac{n-2}{4}}\(n-2\)^{\frac{n+2}{4}}\lambda_{i,\alpha}\left\|\nabla V_i\right\|_2+\o\(\sum_{j=0}^n\left|\lambda_{j,\alpha}\right|\)
\end{equation}
as $\alpha\to+\infty$, where the function $V_i$ is as in \eqref{Eq11}. If $\(t_\alpha,\xi_\alpha\)$ is a critical point of $\mathcal{J}_{\varepsilon_\alpha}$ for all $\alpha$, then it follows from \eqref{Pr2Eq10} and \eqref{Pr2Eq17} that for any $i=0,\dotsc,n$, there $\lambda_{i,\alpha}=0$ for all $i=0,\dotsc,n$. By \eqref{Pr2Eq3}, if follows that for $\alpha$ large, the function $u_0-W_{\delta_{\varepsilon_\alpha}\(t_\alpha\),\xi_\alpha}+\phi_{\delta_{\varepsilon_\alpha}\(t_\alpha\),\xi_\alpha}$ is a critical point of the functional $J_{\varepsilon_\alpha}$, and therefore a solution of equation \eqref{Eq8}. This ends the proof of Proposition~\ref{Pr2}.\endproof

\section{Proof of the theorems}\label{sec:pfs}
\proof[Proof of Theorems~\ref{Th1} and \ref{Th:lcf}]
We let $\mathcal{G}$ be the function defined on $\mathbb{R}_{>0}\times M$ by 
\begin{equation}\label{Th1Eq1}
\mathcal{G}\(t,\xi\):=c_4(n)\ln \frac{1}{t}+c_5(n)t^{\frac{n-2}{2}}u_0\(\xi\),
\end{equation}
where $c_4(n)$ and $c_5(n)$ are as in \eqref{Pr2Eq1}. Since $u_0$ is positive and $M$ is compact, we get
\begin{equation}\label{Th1Eq2}
\lim_{t\to0}\mathcal{G}\(t,\xi\)=+\infty\quad\text{and}\quad\lim_{t\to+\infty}\mathcal{G}\(t,\xi\)=+\infty
\end{equation}
uniformly with respect to $\xi\in M$. Since $h\equiv c_n\Scal_g$ and either $\{3\le n\le9\}$ or $\{(M,g)$ is locally conformally flat$\}$, it follows from Proposition~\ref{Pr2} that
\begin{equation}\label{Th1Eq3}
\lim_{\varepsilon\to0}\frac{1}{\varepsilon}\(\mathcal{J}_\varepsilon\(t,\xi\)-c_1(n,u_0)-c_2(n,u_0)\varepsilon-c_3(n)\varepsilon\ln\varepsilon\)=\mathcal{G}\(t,\xi\)
\end{equation}
uniformly with respect to $t$ in compact subsets of $\mathbb{R}_{>0}$ and with respect to the point $\xi$ in $M$. For $\varepsilon$ small, by \eqref{Th1Eq2}, \eqref{Th1Eq3}, and by continuity of $\mathcal{J}_\varepsilon$ and $\mathcal{G}$, we get the existence of a family of points $\(t_\varepsilon,\xi_\varepsilon\)$ which realize the minimum values of the functions $\mathcal{J}_\varepsilon$ in $\(a,b\)\times M$ for some positive real numbers $a<b$ independent of $\varepsilon$. By Proposition~\ref{Pr2}, it follows that for $\varepsilon$ small, the function $u_\varepsilon=u_0-W_{\delta_\varepsilon\(t_\varepsilon\),\xi_\varepsilon}+\phi_{\delta_\varepsilon\(t_\varepsilon\),\xi_\varepsilon}$ is a solution of equation \eqref{Eq1}, where $W_{\delta_\varepsilon\(t\),\xi}$ is as in \eqref{Eq9} and $\phi_{\delta_\varepsilon\(t\),\xi}$ is given by Proposition~\ref{Pr1}.\par

\medskip\noindent We get that $\lim_{\varepsilon\to 0}u_\varepsilon=u_0$ in $H_{1,loc.}^2(M\setminus\{\xi_0\})$ where $\xi_0:=\lim_{\varepsilon\to 0}\xi_\varepsilon$ (up to a subsequence): it then follows from standard elliptic theory that $\lim_{\varepsilon\to 0}u_\varepsilon=u_0$ in $C^2_{loc}(M\setminus\{\xi_0\})$. Independently, $\lim_{\varepsilon\to 0}\delta_\varepsilon\(t_\varepsilon\)^{\frac{n-2}{2}}u_\varepsilon(\hbox{exp}_{\xi_\epsilon}\cdot)=-U$ in $H_{1,loc.}^2(\rn)$, and still by elliptic theory, one then gets the convergence in $C^2_{loc}(\rn)$. This proves that $(u_\varepsilon)_{\varepsilon>0}$ changes sign and blows-up when $\varepsilon\to 0$. This ends the proof of Theorems~\ref{Th1} and \ref{Th:lcf}.\endproof

\noindent\proof[Proof of Theorems~\ref{Th2} and \ref{Th:dim6}]
In dimensions $3\leq n\le5$, the proof of Theorem \ref{Th2} is similar to the proof of Theorem~\ref{Th1}. The specificity of dimension $n=6$, is that the function $\mathcal{G}$ in \eqref{Th1Eq1} is replaced by
$$\mathcal{G}\(t,\xi\):=c_4(6)\ln \frac{1}{t}+c_5(6)\(u_0\(\xi\)+\frac{1}{2}\(h\(\xi\)-c_6\Scal_g\(\xi\)\)\)t^2,$$
where $c_4(6),\,c_5(6)>0$ are as in \eqref{Pr2Eq1}: therefore \eqref{Th1Eq2} holds with the hypothesis of Theorem \ref{Th2} and the proof of Theorem \ref{Th2} goes as for Theorem \ref{Th1}. We focus on Theorem \ref{Th:dim6}. In dimension $n=6$, computations similar to \eqref{Pr2Eq1} yield
\begin{multline*}
J_\varepsilon^+\(u_0+W_{\delta_\varepsilon\(t\),\xi}\)=c_1(6,u_0)+c_2(6,u_0)\varepsilon+c_3(6)\varepsilon\ln\varepsilon\\
+\bigg(c_4(6)\ln\frac{1}{t}+c_5(6)\(\frac{1}{2}\left(h\(\xi\)-c_6\Scal_g\(\xi\)\right)-u_0(\xi)\)t^2\bigg)\varepsilon+\o\(\varepsilon\)
\end{multline*}
as $\varepsilon\to0$, where $J^+_\varepsilon\(u\):=\frac{1}{2}\int_M\left|\nabla u\right|^2_gdv_g+\frac{1}{2}\int_Mhu^2dv_g-\frac{1}{\crit-\eps}\int_Mu_+^{\crit-\eps}dv_g$. The proof then is similar to the proof of Theorem \ref{Th2}. \endproof
\noindent\proof[Proof of Theorem~\ref{Th:dim10}] The introduction of another type of model for blow-up is required here. It follows from Lee--Parker~\cite{leeparker} that for any $\xi\in M$, there exists $\Lambda_\xi\in C^\infty(M)$ positive such that $g_\xi:=\Lambda_\xi^{\frac{4}{n-2}}g$ satisfies $dv_{g_{\xi}}=(1+O(d_{g_\xi}(\xi,\cdot)^n))\, dx$ in a geodesic normal chart. An immediate consequence is that $\Scal_{g_\xi}(\xi)=|\nabla\Scal_{g_\xi}(\xi)|=0$ and $\Delta_{g_\xi}\Scal_{g_\xi}(\xi)=\frac{1}{6}|\Weyl_g(\xi)|^2_g$. Moreover, we can assume that $(\xi,x)\mapsto \Lambda_\xi(x)$ is $C^\infty$ and $\nabla\Lambda_\xi(\xi)=0$. We define $W_{\delta,\xi}$ in \eqref{Eq9} with the function $\Lambda_\xi$ above. When $h\equiv c_n\Scal_g$, the conformal law of change of metric yields the Taylor expansion
\begin{multline}\label{eq:exp:10}
J_\varepsilon\(u_0-W_{\delta_\varepsilon\(t\),\xi}\)=c_1(n,u_0)+c_2(n,u_0)\varepsilon+c_3(n)\varepsilon\ln\varepsilon
+\frac{K_n^{-n}}{n}\bigg(\frac{(n-2)^2}{4}\eps\ln\frac{1}{t}\\+\frac{2^n\omega_{n-1}}{\omega_n(n(n-2))^{(n-2)/4}}u(\xi)\eps t^{\frac{n-2}{2}}-\frac{|\Weyl_g(\xi)|_g^2}{24(n-4)(n-6)}\eps^{\frac{8}{n-2}}t^4\bigg)+\o\(\varepsilon+\varepsilon^{\frac{8}{n-2}}\)
\end{multline}
when $\eps\to 0$ for $n\geq 7$. When $n<10$, the term involving the Weyl tensor is neglictible. When $n=10$, it competes with the one involving $u_0$: arguing as in the proofs above, we get the existence of a blowing-up family when $u_0>\frac{5}{567}|\Weyl_g|_g^2$, which proves Theorem \ref{Th:dim10} since the additional terms involving $\phi_{\delta_\varepsilon\(t\),\xi}$ are neglictible when $n\leq 17$. When $n>10$, the Weyl tensor dominates but the negative sign does not allow to construct a critical point for the reduced functional.\endproof

\noindent\proof[Proof of Theorem~\ref{Th:presc}] If $\xi_0\in M$ is a strict minimizer of $\Phi$ on $\overline{B}_{\xi_0}(\nu_0)\subset M$ with $\nu_0>0$, the arguments above extend by minimizing $\mathcal{G}$ on $(0,+\infty)\times B_{\xi_0}(\nu_0)$.\endproof

\section{Error estimate}\label{sec:error}
This section is devoted to the error estimate used in previous sections.  All notations refer to Section~\ref{sec:finite}. The estimate is as follows:

\begin{lemma}\label{Lem4}
Given two positive real numbers $a<b$, there exists a positive constant $C'_{a,b}$ such that for $\varepsilon$ small, for any real number $t$ in $\[a,b\]$, and any point $\xi$ in $M$, there holds
\begin{multline}
\left\|i^*\(f_\varepsilon\(u_0-W_{\delta_\varepsilon\(t\),\xi}\)\)-u_0+W_{\delta_\varepsilon\(t\),\xi}\right\|_h\label{Lem4Eq1}\\
\qquad\le C'_{a,b}\left\{\begin{aligned}
&\varepsilon\left|\ln\varepsilon\right|&&\text{if }n\le6\\
&\varepsilon^{\frac{4}{n-2}}&&\text{if }n\ge7\\
&\varepsilon^{\frac{n+2}{2\(n-2\)}}&&\text{if }n\ge7,\,h\equiv c_n\Scal_g,\text{ and }(M,g)\text{ loc. conformally flat},
\end{aligned}\right.
\end{multline}
where $\delta_\varepsilon\(t\)=t\varepsilon^{2/\(n-2\)}$ and $W_{\delta_\varepsilon\(t\),\xi}$ is as in \eqref{Eq9}.
\end{lemma}

\proof
All our estimates in this proof are uniform with respect to $t$ in $\[a,b\]$, $\xi$ in $M$, and $\varepsilon$ in $\(0,\varepsilon_0\)$ for some fixed positive real number $\varepsilon_0$. The continuity of $i^*$ yields
\begin{multline*}
\left\|i^*\(f_\varepsilon\(u_0-W_{\delta_\varepsilon\(t\),\xi}\)\)-u_0+W_{\delta_\varepsilon\(t\),\xi}\right\|_h\\
=\O\(\left\|f_\varepsilon\(u_0-W_{\delta_\varepsilon\(t\),\xi}\)-\(\Delta_g+h\)\(u_0-W_{\delta_\varepsilon\(t\),\xi}\)\right\|_{\frac{2n}{n+2}}\).
\end{multline*}
It follows that
\begin{equation}\label{Lem4Eq3}
\left\|i^*\(f_\varepsilon\(u_0-W_{\delta_\varepsilon\(t\),\xi}\)\)-u_0+W_{\delta_\varepsilon\(t\),\xi}\right\|_h=\O\Big(\tilde{I}_{1,\eps,t,\xi}+\tilde{I}_{2,\eps,t,\xi}+\tilde{I}_{3,\eps,t,\xi}\Big),
\end{equation}
where
$$\tilde{I}_{1,\eps,t,\xi}:=\left\|f_\varepsilon\(u_0-W_{\delta_\varepsilon\(t\),\xi}\)-f_\varepsilon\(u_0\)+f_\varepsilon\(W_{\delta_\varepsilon\(t\),\xi}\)\right\|_{\frac{2n}{n+2}},$$
$$\tilde{I}_{2,\eps,t,\xi}:=\left\|f_\varepsilon\(u_0\)-\Delta_gu_0-hu_0\right\|_{\frac{2n}{n+2}},$$
$$\tilde{I}_{3,\eps,t,\xi}:=\left\|f_\varepsilon\(W_{\delta_\varepsilon\(t\),\xi}\)-\Delta_gW_{\delta_\varepsilon\(t\),\xi}-hW_{\delta_\varepsilon\(t\),\xi}\right\|_{\frac{2n}{n+2}}.$$
We estimate these terms separately.

\medskip\noindent{\it Step 1: Estimate of $\tilde{I}_{1,\eps,t,\xi}$.}\\
\noindent We get
\begin{multline*}
\tilde{I}_{1,\eps,t,\xi}
\le\left\|\(f_\varepsilon\(u_0-W_{\delta_\varepsilon\(t\),\xi}\)+f_\varepsilon\(W_{\delta_\varepsilon\(t\),\xi}\)\)\mathbf{1}_{B_\xi(\sqrt{\delta_\varepsilon\(t\)})}\right\|_{\frac{2n}{n+2}}\\
+\left\|\(f_\varepsilon\(u_0-W_{\delta_\varepsilon\(t\),\xi}\)-f_\varepsilon\(u_0\)\)\mathbf{1}_{M\backslash B_\xi(\sqrt{\delta_\varepsilon\(t\)})}\right\|_{\frac{2n}{n+2}}\\
+\left\|f_\varepsilon\(W_{\delta_\varepsilon\(t\),\xi}\)\mathbf{1}_{M\backslash B_\xi(\sqrt{\delta_\varepsilon\(t\)})}\right\|_{\frac{2n}{n+2}}+\left\|f_\varepsilon\(u_0\)\mathbf{1}_{B_\xi(\sqrt{\delta_\varepsilon\(t\)})}\right\|_{\frac{2n}{n+2}}.
\end{multline*}
As is easily checked, Taylor's expansion for $f_\varepsilon\(u_0-W_{\delta_\varepsilon\(t\),\xi}\)$ yields
\begin{multline*}
\left\|\(f_\varepsilon\(u_0-W_{\delta_\varepsilon\(t\),\xi}\)+f_\varepsilon\(W_{\delta_\varepsilon\(t\),\xi}\)\)\mathbf{1}_{B_\xi(\sqrt{\delta_\varepsilon\(t\)})}\right\|_{\frac{2n}{n+2}}\\
\le C\(\left\|u_0W_{\delta_\varepsilon\(t\),\xi}^{2^*-2-\varepsilon}\mathbf{1}_{B_\xi(\sqrt{\delta_\varepsilon\(t\)})}\right\|_{\frac{2n}{n+2}}+\left\|u_0^{2^*-1-\varepsilon}\mathbf{1}_{B_\xi(\sqrt{\delta_\varepsilon\(t\)})}\right\|_{\frac{2n}{n+2}}\)
\end{multline*}
and
\begin{multline*}
\left\|\(f_\varepsilon\(u_0-W_{\delta_\varepsilon\(t\),\xi}\)-f_\varepsilon\(u_0\)\)\mathbf{1}_{M\backslash B_\xi(\sqrt{\delta_\varepsilon\(t\)})}\right\|_{\frac{2n}{n+2}}\\
\le C\(\left\|u_0^{2^*-2-\varepsilon}W_{\delta_\varepsilon\(t\),\xi}\mathbf{1}_{M\backslash B_\xi(\sqrt{\delta_\varepsilon\(t\)})}\right\|_{\frac{2n}{n+2}}+\left\|W_{\delta_\varepsilon\(t\),\xi}^{2^*-1-\varepsilon}\mathbf{1}_{M\backslash B_\xi(\sqrt{\delta_\varepsilon\(t\)})}\right\|_{\frac{2n}{n+2}}\).
\end{multline*}
Estimating roughly these terms yields
\begin{equation}\label{Lem4Eq11}
\tilde{I}_{1,\eps,t,\xi}=\left\{\begin{aligned}
&\O\(\varepsilon\)&&\text{if }n\le5\\
&\O\(\varepsilon\left|\ln\varepsilon\right|^{\frac{2}{3}}\)&&\text{if }n=6\\
&\O\(\varepsilon^{\frac{n+2}{2\(n-2\)}}\)&&\text{if }n\ge7
\end{aligned}\right.
\end{equation}
when $\eps\to 0$.

\medskip\noindent{\it Step 2: Estimate of $\tilde{I}_{2,\eps,t,\xi}$.}\\
\noindent Since $u_0$ is a solution of \eqref{Th2Eq}, we get that
\begin{equation}\label{Lem4Eq12}
\tilde{I}_{2,\eps,t,\xi}=\left\|f_\varepsilon\(u_0\)-f_0\(u_0\)\right\|_{\frac{2n}{n+2}}=\O\(\varepsilon\).
\end{equation}

\medskip\noindent{\it Step 3: Estimate of $\tilde{I}_{3,\eps,t,\xi}$.}\\
\noindent We define $\chi_\xi\(\cdot\)=\chi\(d_{g_\xi}\(\cdot,\xi\)\)$, $U_{\delta,\xi}\(\cdot\)=\delta^{\frac{2-n}{2}}U(\delta^{-1}\exp_\xi^{-1}(\cdot))$, where the function $U$ is as in \eqref{Eq10} and the exponential map is taken with respect to the metric $g_\xi$. 

\medskip\noindent{\it Step 3.1: Estimate of $I_{3,\eps,t,\xi}$ when $(M,g)$ is locally conformally flat and $h\equiv c_n \Scal_g$.}\\ 
Since $g_\xi=\varLambda_\xi^{4/\(n-2\)}g$ is flat, we get that
$$f_\varepsilon\(W_{\delta_\varepsilon\(t\),\xi}\)-\Delta_gW_{\delta_\varepsilon\(t\),\xi}-hW_{\delta_\varepsilon\(t\),\xi}=\varLambda_\xi^{2^*-1}\(\varLambda_\xi^{-\varepsilon}f_\varepsilon\(\widetilde{W}_{\delta_\varepsilon\(t\),\xi}\)-\Delta_{g_\xi}\widetilde{W}_{\delta_\varepsilon\(t\),\xi}\),$$
where $\widetilde{W}_{\delta_\varepsilon\(t\),\xi}=W_{\delta_\varepsilon\(t\),\xi}/\varLambda_\xi$. In this case, since the metric $g_\xi$ is flat in $B_\xi\(r_0\)$ and since the function $U$ is a solution of the equation $\Delta_{\Eucl}U=U^{2^*-1}$ in $\mathbb{R}^n$, we get that
\begin{multline}\label{Lem4Eq13}
I_{3,\eps,t,\xi}\le\left\|\(\chi_\xi\varLambda_\xi\)^{2^*-1-\varepsilon}\(U^{2^*-1-\varepsilon}_{\delta_\varepsilon\(t\),\xi}-U^{2^*-1}_{\delta_\varepsilon\(t\),\xi}\)\right\|_{\frac{2n}{n+2}}\\
+\left\|\(\chi_\xi^{2^*-1-\varepsilon}\varLambda_\xi^{-\varepsilon}-\chi_\xi\)\varLambda_\xi^{2^*-1}U^{2^*-1}_{\delta_\varepsilon\(t\),\xi}\right\|_{\frac{2n}{n+2}}+\left\|\varLambda_\xi^{\crit-1} U_{\delta_\varepsilon\(t\),\xi}\Delta_{g_\xi}\chi_\xi\right\|_{\frac{2n}{n+2}}\\
+2\left\|\varLambda_\xi^{\crit-1} \<\nabla\chi_\xi,\nabla U_{\delta_\varepsilon\(t\),\xi}\>_{g_\xi}\right\|_{\frac{2n}{n+2}}.
\end{multline}

\medskip\noindent{\it Step 3.2: Estimate of $I_{3,\eps,t,\xi}$ in the general case.}\\
\noindent In general, we get that 
\begin{multline}\label{Lem4Eq14}
I_{3,\eps,t,\xi}\le\left\|\chi_\xi^{2^*-1-\varepsilon}\(U^{2^*-1-\varepsilon}_{\delta_\varepsilon\(t\),\xi}-U^{2^*-1}_{\delta_\varepsilon\(t\),\xi}\)\right\|_{\frac{2n}{n+2}}+\left\|\(\chi_\xi^{2^*-1-\varepsilon}-\chi_\xi\)U^{2^*-1}_{\delta_\varepsilon\(t\),\xi}\right\|_{\frac{2n}{n+2}}\\
+\left\|\chi_\xi\(U_{\delta_\varepsilon\(t\),\xi}^{\crit-1}-\Delta_gU_{\delta_\varepsilon\(t\),\xi}\)\right\|_{\frac{2n}{n+2}}+\left\|U_{\delta_\varepsilon\(t\),\xi}\Delta_g\chi_\xi\right\|_{\frac{2n}{n+2}}\\
+2\left\|\<\nabla\chi_\xi,\nabla U_{\delta_\varepsilon\(t\),\xi}\>_g\right\|_{\frac{2n}{n+2}}+\left\|h\chi_\xi U_{\delta_\varepsilon\(t\),\xi}\right\|_{\frac{2n}{n+2}}.
\end{multline}

\medskip\noindent{\it Step 3.3: Estimates of the terms in \eqref{Lem4Eq13} and \eqref{Lem4Eq14}.}\\
\noindent Since $\chi_\xi\equiv 1$ on $B_{\xi}(r_0/2)$ and  $\chi_\xi\equiv 0$ on $M\setminus B_{\xi}(r_0)$, we get that 
\begin{align}
&\int_M\left|\(\chi_\xi\varLambda_\xi\)^{2^*-1-\varepsilon}\(U^{2^*-1-\varepsilon}_{\delta_\varepsilon\(t\),\xi}-U^{2^*-1}_{\delta_\varepsilon\(t\),\xi}\)\right|^{\frac{2n}{n+2}}dv_g=\O\(\varepsilon^{\frac{2n}{n+2}}\left|\ln\varepsilon\right|^{\frac{2n}{n+2}}\),\label{Lem4Eq15}\\
&\int_M\left|\(\chi_\xi^{2^*-1-\varepsilon}\varLambda_\xi^{-\varepsilon}-\chi_\xi\)\varLambda_\xi^{2^*-1}U^{2^*-1}_{\delta_\varepsilon\(t\),\xi}\right|^{\frac{2n}{n+2}}dv_g=\O\(\varepsilon^{\frac{2n}{n+2}}\),\label{Lem4Eq16}\\
&\int_M\left|\varLambda_\xi^{\crit-1} U_{\delta_\varepsilon\(t\),\xi}\Delta_g\chi_\xi\right|^{\frac{2n}{n+2}}=\O\(\varepsilon^{\frac{2n}{n+2}}\)\label{Lem4Eq18}
\end{align}
when $\eps\to 0$. A rough $L^\infty$ upper bound for $|\nabla U_{\delta_\eps(t),\xi}|$ on $M\setminus B_{r_0/2}(\xi)$ yields 
\begin{align}
&\int_M\left|\<\varLambda_\xi^{\crit-1}\nabla\chi_\xi,\nabla U_{\delta_\varepsilon\(t\),\xi}\>_g\right|^{\frac{2n}{n+2}}dv_g=\O\(\varepsilon^{\frac{2n}{n+2}}\).\label{Lem4Eq19}
\end{align}
Since $\Delta_{\Eucl}U=U^{\crit-1}$, we get in the chart $\hbox{exp}_{\xi}$ that
$$\Delta_g U_{\delta_\varepsilon\(t\),\xi}-U_{\delta_\varepsilon\(t\),\xi}^{\crit-1}=-(g^{ij}-\delta^{ij})\partial_{ij}U_{\delta_\varepsilon\(t\),\xi}+g^{ij}\Gamma_{ij}^k\partial_k U_{\delta_\varepsilon\(t\),\xi}\,,$$
where the $g^{ij}$'s are the coordinate of the metric $g=g_\xi$ and $\Gamma^\gamma_{\alpha\beta}$'s are the Christoffel symbols of the metric $g$ in the normal chart $\hbox{exp}_{\xi}$. Cartan's expansion of the metric yields $|g^{ij}(x)-\delta^{ij}|\leq C|x|^2$ and $|\Gamma^k_{ij}(x)|\leq C|x|$ around $0$, and therefore
$$\left|U_{\delta_\varepsilon\(t\),\xi}^{\crit-1}-\Delta_g U_{\delta_\varepsilon\(t\),\xi}\right|\leq C|x|^2|\nabla^2 U_{\delta_\varepsilon\(t\),\xi}|+C|x|\cdot |\nabla U_{\delta_\varepsilon\(t\),\xi}|$$
via the chart $\hbox{exp}_{\xi}$. Therefore, we get that
\begin{equation}\label{Lem4Eq17}
\int_M\left|\chi_\xi\(U_{\delta_\varepsilon\(t\),\xi}^{\crit-1}-\Delta_gU_{\delta_\varepsilon\(t\),\xi}\)\right|^{\frac{2n}{n+2}}dv_g=\left\{\begin{aligned}
&\O\(\varepsilon^{\frac{2n}{n+2}}\)&&\text{if }n\le5\\
&\O\(\varepsilon^{\frac{3}{2}}\left|\ln\varepsilon\right|\)&&\text{if }n=6\\
&\O\(\varepsilon^{\frac{8n}{\(n+2\)\(n-2\)}}\)&&\text{if }n\ge7.
\end{aligned}\right.
\end{equation}
It remains to compute
\begin{equation}
\int_M\left|h\chi_\xi U_{\delta_\varepsilon\(t\),\xi}\right|^{\frac{2n}{n+2}}dv_g=\left\{\begin{aligned}
&\O\(\varepsilon^{\frac{2n}{n+2}}\)&&\text{if }n\le5\\
&\O\(\varepsilon^{\frac{3}{2}}\left|\ln\varepsilon\right|\)&&\text{if }n=6\\
&\O\(\varepsilon^{\frac{8n}{\(n+2\)\(n-2\)}}\)&&\text{if }n\ge7
\end{aligned}\right.\label{Lem4Eq20}
\end{equation}
when $\eps\to 0$.

\medskip\noindent{\it Step 3.4: End of estimate of $\tilde{I}_{3,\eps,t,\xi}$.}\\
\noindent By \eqref{Lem4Eq14}--\eqref{Lem4Eq20}, we get
\begin{equation}\label{Lem4Eq21}
\left\|f_\varepsilon\(W_{\delta_\varepsilon\(t\),\xi}\)-\Delta_gW_{\delta_\varepsilon\(t\),\xi}-hW_{\delta_\varepsilon\(t\),\xi}\right\|_{\frac{2n}{n+2}}=\left\{\begin{aligned}
&\O\(\varepsilon\left|\ln\varepsilon\right|\)&&\text{if }n\le6\\
&\O\(\varepsilon^{\frac{4}{n-2}}\)&&\text{if }n\ge7.
\end{aligned}\right.
\end{equation}
In case $h\equiv c_n\Scal_g$ and the manifold is locally conformally flat, by \eqref{Lem4Eq13}, \eqref{Lem4Eq15}--\eqref{Lem4Eq19}, we get
\begin{equation}\label{Lem4Eq22}
\left\|f_\varepsilon\(W_{\delta_\varepsilon\(t\),\xi}\)-\Delta_gW_{\delta_\varepsilon\(t\),\xi}-hW_{\delta_\varepsilon\(t\),\xi}\right\|_{\frac{2n}{n+2}}=\O\(\varepsilon\left|\ln\varepsilon\right|\)
\end{equation}
when $\eps\to 0$.

\medskip\noindent{\it Step 4: End of proof of \eqref{Lem4Eq1}.}\\
\noindent Finally, \eqref{Lem4Eq1} follows from \eqref{Lem4Eq3}, \eqref{Lem4Eq11}, \eqref{Lem4Eq12}, \eqref{Lem4Eq21}, and \eqref{Lem4Eq22}.
\endproof

\end{document}